\documentclass[leqno,twoside]{amsart}

\usepackage{latexsym,esint}
\usepackage{psfig}

\theoremstyle{plain}

\def\endproof{\hspace*{\fill}\mbox{\ \rule{.1in}{.1in}}\medskip }

\newtheorem{theorem}{Theorem}[section]
\newtheorem{corollary}[theorem]{Corollary}
\newtheorem{lemma}[theorem]{Lemma}
\newtheorem{proposition}[theorem]{Proposition}
\newtheorem{definition}[theorem]{Definition}

\theoremstyle{definition}

\newtheorem{remark}[theorem]{Remark}

\numberwithin{equation}{section}
\numberwithin{figure}{section}

\begin{document}

\title[nonlinear shell theories]
{Shell theories arising as low energy $\mathbf{\Gamma}$-limit \\ of 3d nonlinear elasticity}
\author{Marta Lewicka, Maria Giovanna Mora and Mohammad Reza Pakzad}
\address{Marta Lewicka, University of Minnesota, Department of Mathematics, 
206 Church St. S.E., Minneapolis, MN 55455}
\address{Maria Giovanna Mora, Scuola Internazionale Superiore di Studi Avanzati,
via Beirut 2-4, 34014 Trieste, Italy}
\address{Mohammad Reza Pakzad, University of Pittsburgh, Department of Mathematics, 
139 University Place, Pittsburgh, PA 15260}
\email{lewicka@math.umn.edu, mora@sissa.it, pakzad@pitt.edu}
 
\begin{abstract} 
We discuss the limiting behavior (using the notion of $\Gamma$-limit) of the
3d nonlinear elasticity for thin shells around an arbitrary smooth 2d 
surface.  In particular, under the assumption that the elastic energy
of deformations scales like $h^4$, $h$ being the thickness of a shell,  
we derive a limiting theory which is a generalization of the von K\'arm\'an 
theory for plates.
\end{abstract}

\maketitle
\tableofcontents

\section{Introduction}

The derivation of lower dimensional models for thin structures (such as membranes, 
shells, or beams) from the three-dimensional theory, has been one of the fundamental 
questions since the beginning of research in elasticity \cite{Love}.
Recently, a novel variational approach through {\em $\Gamma$-convergence}
has lead to the derivation of a hierarchy of limiting theories.
Among other features, it provides a rigorous justification
of convergence of three-dimensional minimizers to minimizers 
of suitable lower dimensional limit energies.

In this paper we discuss shell theories arising as $\Gamma$-limits of {\em higher scalings} of the
{\em nonlinear} elastic energy. Given a $2$-dimensional surface 
$S$, consider a shell $S^h$ of mid-surface $S$ and thickness $h$, 
and associate to its deformation $u$ the scaled per unit thickness
three dimensional nonlinear elastic energy $E^{elastic}(u, S^h)$.
We are interested in the identification of the $\Gamma$-limit $I_\beta$ 
of the energies:
\begin{equation*}
h^{-\beta} E^{elastic}(\cdot, S^h),
\end{equation*}
as $h\to 0$, for a given scaling $\beta\geq 0$. As mentioned above, 
this implies convergence, in a suitable sense, of minimizers $u^h$ of $E^{elastic}$ 
(subject to applied forces) to minimizers of two-dimensional energy $I_\beta$, 
provided $E^{elastic}(u^h, S^h)\leq Ch^\beta.$  

In the case when $S$ is a subset of $\mathbb R^2$ (i.e. a plate), 
such $\Gamma$-convergence was first established by
LeDret and Raoult \cite{LR1} for $\beta = 0$, then
by Friesecke, James and M\"uller \cite{FJMgeo, FJMhier} for all $\beta \geq 2$ (see also
\cite{Pa03} for results for $\beta=2$ under additional
conditions). In the case of $0 < \beta < 5/3$, the convergence was recently 
obtained by Conti and Maggi \cite{CM05}, see also \cite{Co03}. 
The regime $5/3 \le \beta<2$ remains open and is conjectured to be relevant 
for the crumpling of elastic sheets.  Other significant results 
for plates concern derivation of limit theories for incompressible materials 
\cite{CD1, CD2, T}, for heterogeneous materials \cite{Schmidt} and through establishing 
convergence of equilibria, rather than strict minimizers \cite{Mo03, MuPa}. 


Much less is known in the general case when $S$ is an arbitrary surface.
The first result by LeDret and Raoult \cite{LeD-Rao} 
relates to scaling $\beta=0$ and models {\em membrane shells}: the limit $I_0$
depends only on the stretching and shearing produced by the deformation on the mid-surface $S$. 
Another study is due to Friesecke, James, Mora and M\"uller \cite{FJMM_cr}, who analyzed 
the case $\beta=2$.  This scaling corresponds to a {\em flexural shell model}, 
where the only admissible deformations are those preserving the metric on $S$. 
The energy $I_2$ depends then on the change of curvature produced by the deformation. 

All the above mentioned theories (as well as the subsequent results in this paper) 
should be put in contrast with a large body of literature, devoted to derivations 
starting from three-dimensional {\em linear} elasticity (see Ciarlet \cite{ciarbookvol3} 
and references therein).
Indeed, since thin structures may undergo large rotations even under the action 
of very small forces, one cannot assume the small strain condition, 
on which the linear elasticity is based.

\medskip

The objective of this work is to discuss the limit energies for scalings $\beta\geq 4$,
for arbitrary surfaces $S$.  We now give a heuristic overview of our results, 
whose precise formulations will be presented in section 2. 
If $E^{elastic}(u, S^h)\approx Ch^\beta$, for any  $\beta>2$, one expects $u$ to be close 
to a rigid motion $R$. This argument can be made precise by means of the quantitative 
rigidity estimate due to  Friesecke, James and M\"uller \cite{FJMgeo} (see also Lemma \ref{approx}). 
We further demonstrate that the first term in the expansion of $u-R$, in terms of $h$, 
equals an infinitesimal isometry 
$V$. That is, there is no first order change in the Riemannian metric of $S$ under 
the displacement $V$. The corresponding bending energy, given in terms of 
the first order change in the second fundamental form of $S$, is the $\Gamma$-limit $I_\beta$ 
if $\beta >4$ (Theorem \ref{thmaintre}). This limit energy coincides with the so-called 
{\em linearly elastic flexural shell model}, derived in \cite{ciarbookvol3} from the linear 
elasticity theory. 
Our result guarantees therefore that, without any a priori smallness assumption on the strain, 
the use of the linearized flexural shell model is justified whenever the order of magnitude 
of the per unit thickness three-dimensional energy is $h^\beta$ with $\beta>4$.

When $\beta =4$, also the second order in $h$ change in the metric on $S$
(stretching) contributes to the limiting energy. This change is induced by $V$, 
and additionally, by an ``approximate second order displacement'' $w$. 
This last notion involves studying the {\em finite strain space}. For a similar situation where 
this space emerges see the discussion by Sanchez-Palencia \cite{sanchez} and Geymonat 
and Sanchez-Palencia \cite{GSP} under the title of ill-inhibited surfaces, 
in the context of linear elasticity. 
In Theorems \ref{thmainuno} and \ref{thmaindue} we derive the energy functional $I_4$, which  
can be seen as a generalization of the {\em von K\'arm\'an  theory} for plates \cite{karman}, 
justified in terms of $\Gamma$-convergence in \cite{FJMhier}.  
Indeed, if $S$ is a plate, then $V$ and $w$ are, respectively, the out-of-plane and 
the in-plane displacements (modulo a possible in-plane infinitesimal 
rigid motion).

A particular class of surfaces when $I_4$ simplifies to the bending energy is the hereby introduced
class of {\em approximately robust} surfaces. We say that $S$ is (approximately) robust if 
any infinitesimal isometry $V$ can be completed by a second order displacement to an (approximate) 
second order isometry. In other words, $S$ can always further adjust its deformation, 
to compensate for the change of metric produced at second order. 
As a result, the total stretching of second order 
is insignificant and the $\Gamma$-limit consists only  of a bending term (Theorem \ref{thmaintre}). 
We show three general examples of approximately robust surfaces:
convex surfaces, surfaces of revolution, and developable surfaces without flat parts. An example 
of a not approximately robust surface is a plate.
 
We also address the issue of external forces, depending on the reference configuration, namely the 
{\em dead loads} (Theorem \ref{thmaincinque}). Under a vanishing average condition 
and a suitable scaling of the forces $f^h$ applied to $S^h$, Theorem \ref{thmainuno} 
provides an information of the deformation of $S^h$ assumed in response to the load. 
In addition, the appropriate limit force $f$ identifies the set of possible rotations 
the body will undergo. 
This phenomenon is easily observed: if $f^h$ is ``compressive'', then $S^h$ prefers to make 
a large rotation rather than undergoing a compression, and an alignment of the infinitesimal isometry 
$V$ with the force is energetically preferable.

\medskip

The identification of $\Gamma$-limit for any scaling in the range $\beta\in(2,4)$ 
and non-flat $S$ is still open.
In analogy  with the analysis developed in \cite{FJMhier} for plates, 
the construction of a recovery sequence requires finding an exact isometry of $S$, 
coinciding with a given second order isometry. 
Another direction of study concerns shells, whose mid-surface is inhibited (or infinitesimally rigid). 
Examples of such are closed or partially clamped elliptic surfaces. 
In this case the limit functionals that our theory yields are identically equal to zero.  
This suggests looking for higher order terms in the development of the three-dimensional 
energy in the sense of $\Gamma$-convergence.  These are subtle issues and we plan to address them
in a forthcoming paper.







\bigskip

\noindent{\bf Acknowledgments.} 
We thank Stefan M\"uller for helpful discussions.
A large part of this work was carried out while the second author was visiting
the Institute for Mathematics and its Applications in Minneapolis (USA),
whose support is gratefully acknowledged.
M.L. was partially supported by the NSF grant DMS-0707275.
M.G.M. was partially supported by the Italian Ministry of University and Research
through the project ``Variational problems with multiple scales'' 2006
and by GNAMPA through the project ``Problemi di riduzione di dimensione
per strutture elastiche sottili'' 2008.

\section{An overview of  the main results}

Let $S$ be a $2$-dimensional surface embedded in $\mathbb{R}^3$.
We assume that $S$ is compact, connected, oriented, and of class $\mathcal{C}^{1,1}$,
and that its boundary $\partial S$ is the union of finitely many 
(possibly none) Lipschitz continuous curves.
Consider a family $\{S^h\}_{h>0}$ of thin shells of thickness $h$ around $S$:
$$S^h = \{z=x + t\vec n(x); ~ x\in S, ~ -h/2< t < h/2\}.$$
We will use the following notation:
$\vec n(x)$ for the unit normal, $T_x S$ for the tangent space,
and $\Pi(x) = \nabla \vec n(x)$ for the shape operator on $S$, at a given $x\in S$.
The projection onto $S$ along $\vec n$ will be denoted by $\pi$, so that:
$$\pi(z) = x \qquad \forall z=x+t\vec n(x)\in S^h.$$
We will assume that $h<h_0$, with $h_0>0$ sufficiently small
to have $\pi$ well defined, and so that:
$1/2<|\mbox{Id } + t\Pi(x)|<3/2$ for all $z$ as above.

\medskip

For a $W^{1,2}$ deformation of a thin shell $u^h: S^h\longrightarrow \mathbb{R}^3$,
we assume that its elastic energy (scaled per unit thickness) 
is given by the nonlinear functional:
$$E^{elastic}(u^h, S^h) = \frac{1}{h}\int_{S^h} W(\nabla u^h),$$
where the stored-energy density function $W$ is nonnegative and $\mathcal{C}^2$ 
in some open neighborhood $\mathcal{O}$ of $SO(3)$, in the space 
$\mathbb{R}^{3\times 3}$ of $3\times  3$ real matrices.
Moreover, $W$ is assumed to satisfy, for all $F\in \mathbb{R}^{3\times 3}$ 
and some $C>0$:
\begin{equation*}
\begin{split}
&W(RF) = W(F)\qquad \forall R\in SO(3),\\
&W(R) = 0 \qquad \forall R\in SO(3),\\
&W(F)\geq C \mathrm{dist}^2(F, SO(3)).  
\end{split}
\end{equation*}
Here $SO(3)$ denotes the group of proper rotations.
Recall that the tangent space to $SO(3)$ at $\mbox{Id}$ is the space
of skew-symmetric matrices:
$$so(3) = \{F\in\mathbb{R}^{3\times 3}; ~~ F=-F^T\}.$$

\medskip

It is convenient to view $u^h$ through their rescalings 
$y^h\in W^{1,2}(S^{h_0}, \mathbb{R}^3)$:
\begin{equation*}
y^h(x + t\vec n(x)) = u^h\left(x + t{h}/{h_0}\vec n(x)\right) 
\qquad \forall x\in S \quad \forall t\in (-h_0/2, h_0/2).
\end{equation*}
The advantage is that all $y^h$ have the common domain $S^{h_0}$.
We are concerned with the limiting behavior of the energies:
\begin{equation*}
I^h(y^h) = \frac{1}{h} \int_{S^h}W(\nabla u^h),
\end{equation*}
relative to low energy deformations.  That is, we want to discuss the limit,
as $h\to 0$, of the functionals $I^h/e^h$, for a given sequence
of positive numbers $e^h$, which we assume to satisfy:
\begin{equation}\label{ehass}
\lim_{h\to 0} e^h/h^4 = \kappa^2 <\infty.
\end{equation}
With this in mind, define the related
scaled average displacement:
$$(V^h[y^h])(x) = \frac{h}{\sqrt{e^h}} \fint_{-h_0/2}^{h_0/2}
y^h(x+t\vec n) - x~\mbox{d}t.$$
Since we will frequently deal with such vector fields $V\in W^{1,2}(S,\mathbb{R}^3)$ on the surface, 
we introduce the following notation. By
$\mbox{sym } \nabla V(x)$ we mean a bilinear form on $T_x S$ given by:
$(\mbox{sym } \nabla V(x) \tau)\eta = \frac{1}{2} [(\partial_\tau V(x))\eta +
(\partial_\eta V(x))\tau]$, for all $\tau,\eta\in T_x S$. 
Also, given a matrix field $A\in L^2(S, \mathbb{R}^{3\times 3})$, 
by $A_{tan}(x)$ we denote the tangential minor of $A$ at 
$x\in S$, that is $[(A(x)\tau)\eta]_{\tau,\eta\in T_x S}$.

\medskip

Our first main result is the following:

\begin{theorem}\label{thmainuno}
Assume (\ref{ehass}) and
let $u^h\in W^{1,2}(S^h,\mathbb{R}^3)$ be a sequence of deformations such that
the sequence of scaled energies 
$\{\frac{1}{e^h} I^h(y^h)\}$ is bounded. Then there exist rigid motions of 
$\mathbb{R}^3$, given through proper rotations $Q^h\in SO(3)$ and translations 
$c^h\in\mathbb{R}^3$ such that for the normalized deformations:
$$\tilde y^h(x+t\vec n) = (Q^h)^T y^h(x+t\vec n) - c^h$$
the following holds.
\begin{itemize}
\item[(i)] $\tilde y^h$ converge in $W^{1,2}(S^{h_0})$ to $\pi$.
\item[(ii)] $V^h[\tilde y^h]$ converge (up to a subsequence) in $W^{1,2}(S)$
to some vector field $V\in W^{2,2}(S,\mathbb{R}^3)$ with skew-symmetric gradient, 
that is:
\begin{equation}\label{Vass1}
\partial_\tau V(x) = A(x)\tau \qquad \forall \tau\in T_x S, \,\, {\rm{a.e.}} \,\, x\in S
\end{equation}
for some matrix field $A\in W^{1,2}(S,\mathbb{R}^{3\times 3})$ such that:
\begin{equation}\label{Vass2}
A(x) \in so(3) \qquad \forall x\in S.
\end{equation}
\item[(iii)] $\frac{1}{h} \mbox{sym } \nabla V^h[\tilde y^h]$ converge
(up to a subsequence) weakly in $L^{2}(S)$ to some symmetric matrix field $B_{tan}$ 
on $S$. 
\item[(iv)] There holds:
$$\liminf_{h\to 0} \frac{1}{e^h} I^h(y^h) \geq I(V,B_{tan}),$$
where:
\begin{equation}\label{vonKarman}
I(V,B_{tan})=
\frac{1}{2} 
\int_S \mathcal{Q}_2\left(x,B_{tan} - \frac{\kappa}{2} (A^2)_{tan}\right)
+ \frac{1}{24} \int_S \mathcal{Q}_2\left(x,(\nabla(A\vec n) - A\Pi)_{tan}\right).
\end{equation}
\end{itemize}
\end{theorem}
The following quadratic, nondegenerate forms are of relevance here:
\begin{equation}\label{Qmatrices}
\mathcal{Q}_3(F) = D^2 W(\mbox{Id})(F,F), \qquad
\mathcal{Q}_2(x, F_{tan}) = \min\{\mathcal{Q}_3(\tilde F); ~~ (\tilde F - F)_{tan} = 0\}.
\end{equation}
The form $\mathcal{Q}_3$ is defined for  $F\in\mathbb{R}^{3\times 3}$, while $\mathcal{Q}_2(x,\cdot)$,  
for a given $x\in S$ is defined on tangential minors $F_{tan}$ of such matrices.
Both forms $\mathcal{Q}_3$ and all $\mathcal{Q}_2(x,\cdot)$ are positive definite and depend only on the 
symmetric parts of their arguments (see \cite{FJMgeo}).

Theorem \ref{thmainuno} will be proved in sections \ref{sec_rigidity} 
and \ref{sec_proofthuno}. One of the crucial ingredients is a result on approximating
large deformations \cite{FJMgeo}. For completeness, we sketch its proof, in the setting
of shells, in Appendix A.  We also note that because of the non trivial geometry of the shell, 
the limiting energy $I$, in general, exhibits a dependence on the point $x\in S$, 
although the three-dimensional configuration is homogeneous.

\medskip

Our second main result concerns the possibility of recovering the functional 
$I (V, B_{tan})$
in (\ref{vonKarman}) (or its components), as the limit of scaled energies
$\frac{1}{e^h} I^h(y^h)$, for some sequence of deformations.
For this, define the finite strain displacement space:
$$\mathcal{B} = \Big\{\lim_{h\to 0}\mathrm{sym }~\nabla w^h; ~~ w^h\in W^{1,2}(S,\mathbb{R}^3)\Big\},$$
where limits are taken in $L^2(S)$ (clearly, both the weak and the strong convergences yield the same
$\mathcal{B}$). We then have:
\begin{theorem}\label{thmaindue}
Assume (\ref{ehass}). 
For every $V\in W^{2,2}(S,\mathbb{R}^3)$ satisfying (\ref{Vass1}) and (\ref{Vass2}),
and every $B_{tan}\in\mathcal{B}$, there exists a sequence of deformations
$u^h\in W^{1,2}(S^{h},\mathbb{R}^3)$ such that:
\begin{itemize}
\item[(i)] $y^h$ converge in $W^{1,2}(S^{h_0})$ to $\pi$.
\item[(ii)] $V^h[y^h]$ converge in $W^{1,2}(S)$ to $V$.
\item[(iii)] $\frac{1}{h}\mathrm{sym }~\nabla V^h[y^h]$ converge in $L^{2}(S)$ to $B_{tan}$.
\item[(iv)] Recalling the definition (\ref{vonKarman}) one has:
$$\lim_{h\to 0} \frac{1}{e^h} I^h(y^h) = I(V, B_{tan}).$$
\end{itemize}
\end{theorem}

The form of the limiting energy functional $I$ simplifies, when the space
$\mathcal{B}$ is large enough to choose $B_{tan}$ so that the first term in
(\ref{vonKarman}) vanishes.  That is, we call $S$ ``approximately robust''
if for every $V\in W^{2,2}(S,\mathbb{R}^3)$ satisfying (\ref{Vass1})
(\ref{Vass2}), one has $(A^2)_{tan}\in\mathcal{B}$.  Then we have:

\begin{theorem}\label{thmaintre}
Assume (\ref{ehass}). Let $\kappa = 0$ or let $S$ be approximately robust. 
Then for every 
$V\in W^{2,2}(S,\mathbb{R}^3)$ satisfying (\ref{Vass1}) and (\ref{Vass2}),
there exists a sequence of deformations
$u^h\in W^{1,2}(S^{h},\mathbb{R}^3)$ such that (i) and (ii) of Theorem \ref{thmaindue} hold.
Moreover:
\begin{equation*}
\lim_{h\to 0} \frac{1}{e^h} I^h(y^h) = \tilde I(V),
\end{equation*}
where
\begin{equation}\label{beta>4} 
\tilde I(V) = \frac{1}{24} 
\int_S \mathcal{Q}_2\Big(x,\big(\nabla(A\vec n) - A\Pi\big)_{tan}\Big).
\end{equation}
\end{theorem}
Theorems \ref{thmaindue} and \ref{thmaintre} will be proved in section 
\ref{sec_recovery}.
In section \ref{sec_robust} we discuss the space $\mathcal{B}$ and 
approximately robust surfaces.  In particular, we shall see 
that convex surfaces, surfaces of revolution, 
and non-flat developable surfaces are approximately robust.

\medskip

Theorems \ref{thmainuno} and \ref{thmaindue} (or \ref{thmaintre}) can be summarized
(although they provide more information than the below statement), using the 
language of $\Gamma$-convergence. For completeness, the following result will be 
presented in Appendix B.

\begin{corollary}\label{thmainquattro} 
Assume (\ref{ehass}).
\begin{itemize}
\item[(i)] Define a sequence of functionals:
\begin{equation*} 
\begin{split}
&\mathcal{F}^h: W^{1,2}(S^{h_0},\mathbb{R}^3)
\times W^{1,2}(S,\mathbb{R}^3)\times L^2(S, \mathbb{R}^{2\times 2})
\longrightarrow \overline{\mathbb{R}}\\
&\mathcal{F}^h(y^h, V^h, B^h_{tan}) = \left\{\begin{array}{ll}
\displaystyle{\frac{1}{e^h}I^h(y^h)} & \mbox{ if }~ V^h = V^h[y^h] \mbox{ and } 
B^h_{tan} = \frac{1}{h}\mbox{sym }\nabla V^h,\\
+\infty & \mbox{ otherwise.}
\end{array}\right.
\end{split}
\end{equation*}
Then $\mathcal{F}^h$ $\Gamma$-converge, as $h\to 0$, to the following functional:
\begin{equation*} 
\mathcal{F}(y, V, B_{tan}) = \left\{\begin{array}{ll}
I(V, B_{tan}) & \mbox{ if }~ y=\pi, ~ V\in W^{2,2} \mbox{ satisfies (\ref{Vass1}),
(\ref{Vass2})}, \\
& \mbox{ and } B_{tan}\in\mathcal{B},\\
+\infty & \mbox{ otherwise.}
\end{array}\right.
\end{equation*}
\item[(ii)] Assume that $\kappa=0$ or let $S$ be approximately robust.  
Define the functionals:
\begin{equation*} 
\begin{split}
&\tilde{\mathcal{F}}^h: W^{1,2}(S^{h_0},\mathbb{R}^3)
\times W^{1,2}(S,\mathbb{R}^3)\longrightarrow \overline{\mathbb{R}}\\
&\tilde{\mathcal{F}}^h(y^h, V^h) = \left\{\begin{array}{ll}
\displaystyle{\frac{1}{e^h}I^h(y^h)} & \mbox{ if }~ V^h = V^h[y^h] \\
+\infty & \mbox{ otherwise.}
\end{array}\right.
\end{split}
\end{equation*}
Then $\tilde{\mathcal{F}}^h$ $\Gamma$-converge, as $h\to 0$, to the functional:
\begin{equation*} 
\tilde{\mathcal{F}}(y, V) = \left\{\begin{array}{ll}
\displaystyle{\tilde I(V)}
& \mbox{ if }~ y=\pi \mbox{ and } V\in W^{2,2} \mbox{ satisfies (\ref{Vass1}),
(\ref{Vass2})},\\
+\infty & \mbox{ otherwise.}
\end{array}\right.
\end{equation*}
\end{itemize}
All statements above remain valid if the product spaces (the domains of functionals
$\mathcal{F}^h$, $\tilde{\mathcal{F}}^h$) are equipped with the weak (instead of strong)
topology.
\end{corollary}

We further consider a sequence of forces $f^h\in L^2(S^h,\mathbb{R}^3)$, 
acting on thin shells $S^h$. For simplicity, we assume that:
$$ f^h(x+t\vec n(x)) = h\sqrt{e^h}\det\left(\mbox{Id} + t\Pi(x)\right)^{-1} f(x),$$
where $f\in L^2(S,\mathbf{R}^3)$ is normalized so that:
\begin{equation}\label{fhass}
\int_S f = 0. 
\end{equation}
Define $m$ to be the maximized action of force $f$ on $S$ over all rotations
of $S$, and let $\mathcal{M}$ be the corresponding set of maximizers:
\begin{equation}\label{setM}
\mathcal{M} = \left\{\bar{Q}\in SO(3); ~~ \int_S f(x)\cdot \bar{Q}x~\mbox{d}x = m = 
\max_{Q\in SO(3)} \int_S f\cdot Qx \right\}.
\end{equation}
The total energy functional on $S^h$ is given through:
$$J^h(y^h) = I^h(y^h) + m^h - \frac{1}{h}\int_{S^h} f^h u^h,$$
where $m^h = h\sqrt{e^h} m$.
\begin{theorem}\label{thmaincinque}
Assume (\ref{ehass}) and (\ref{fhass}). Then:
\begin{itemize}
\item[(i)] For every small $h>0$ one has:
$$0\geq \inf\left\{\frac{1}{e^h} J^h(y^h); ~~ u^h\in W^{1,2}(S^h, \mathbb{R}^3)
\right\}\geq -C.$$
\item[(ii)]  If $u^h\in W^{1,2}(S, \mathbb{R}^3)$ is a minimizing sequence
of $\frac{1}{e^h} J^h$, that is:
\begin{equation}\label{min_seq}
\lim_{h\to 0} \left(\frac{1}{e^h}J^h(y^h) - \inf\frac{1}{e^h}J^h \right)=0,
\end{equation}
then there exists $Q^h\in SO(3)$ and $c^h\in\mathbb{R}^3$
such that for the normalized deformations $\tilde y^h = (Q^h)^Ty^h - c^h$
the convergences of Theorem \ref{thmainuno} (i) (ii) and (iii) hold.
The convergence of (a subsequence of) 
$\frac{1}{h}\mbox{sym}\nabla V^h[\tilde y^h]$ to $B_{tan}$ in (iii)
is strong in $L^2(S)$.

Moreover, the set of accumulation points of $\{Q^h\}$ is contained within $\mathcal{M}$.
Any limit $(V, B_{tan}, \bar Q)$ minimizes the functional: 
$$J(V, B_{tan},\bar Q) = I(V, B_{tan}) - \int_S f\cdot \bar{Q} V,$$ 
over all $V\in W^{2,2}(S,\mathbb{R}^3)$ satisfying (\ref{Vass1}) (\ref{Vass2}),
all $B_{tan}\in\mathcal{B}$ and  $\bar Q\in \mathcal{M}$.
\item[(iii)] If $\kappa=0$ in (\ref{ehass}), or if $S$ is approximately robust, 
then for any minimizing sequence as in (\ref{min_seq}), we obtain 
convergences of $\tilde y^h$, $V^h[\tilde y^h]$ and $Q^h$
as described in (ii) above, and the limit $(V,\bar{Q})$ 
minimizes the functional:
$$\tilde J(V,\bar Q) = \tilde I(V) - \int_S f\cdot \bar Q V$$ 
over all $V\in W^{2,2}(S,\mathbb{R}^3)$ satisfying (\ref{Vass1}) (\ref{Vass2})
and all $\bar Q\in \mathcal{M}$.
\end{itemize}
\end{theorem}
In section \ref{sec_Gamma}, we prove Theorem \ref{thmaincinque} and explain 
the significance of the set $\mathcal{M}$ in the setting of dead loads.

The lower bound on the functionals $J$ and $\tilde J$, 
as well as attainment of their infima, can be proved independently,
under conditions (\ref{fhass}) and:
\begin{equation}\label{linear}
\int_S  f(x)\cdot \bar{Q}Fx ~\mbox{d}x = 0 
\quad \forall \bar{Q}\in\mathcal{M}\quad \forall F\in so(3). 
\end{equation}
Here $\mathcal{M}$ is any closed, nonempty subset of $SO(3)$.  When $\mathcal{M}$
has the form (\ref{setM}), then (\ref{linear}) follows from (\ref{setM})
and can be seen as its linearization.
This analysis will be carried out in Appendix C.

\section{Convergence of low energy deformations}\label{sec_rigidity}

In this section we derive some bounds on families of vector mappings $\{u^h\}_{h>0}$,
defined on $S^h$, under the assumption of smallness on their energy.
In what follows, by $C$ we denote an arbitrary positive constant, 
depending on the geometry of $S$ but not on $h$ or the vector mapping under consideration.
In all proofs, the convergences are understood up to a subsequence, 
unless stated otherwise.

\medskip

We will work under the following hypothesis:

\begin{equation*} \mathbf{(H)}\left[~~~~~~~~~~~~~~~
\begin{minipage}{11cm}
A sequence of vector mappings $u^h\in W^{1,2}(S^h, \mathbb{R}^3)$
and a sequence of positive numbers $e^h$ satisfy, for small $h>0$:
\begin{itemize}
\item[(i)] $\displaystyle \frac{1}{h} \int_{S^h}W(\nabla u^h) \leq C e^h$,
\item[(ii)] $\displaystyle \lim_{h\to 0} e^h/h^2 = 0$.
\end{itemize}
\end{minipage}\right.
\end{equation*}

\medskip

As for the flat case in \cite{FJMhier}, the first crucial step is the following 
approximation result:
\begin{lemma}\label{appr}
For each $u^h$ as in {\bf (H)}
there exist a matrix field $R^h\in W^{1,2}(S, \mathbb{R}^{3\times 3})$ such that:
$$R^h(x) \in SO(3) \qquad \forall x\in S,$$
and a matrix $Q^h\in SO(3)$ such that:
\begin{itemize}
\item[(i)] $\displaystyle \|\nabla u^h - R^h\pi\|_{L^2(S^h)} \leq C h^{1/2}\sqrt{e^h},$
\item[(ii)] $\|\nabla R^h\|_{L^2(S)} \leq Ch^{-1} \sqrt{e^h},$
\item[(iii)] $\|(Q^h)^TR^h - \mathrm{Id}\|_{L^p(S)} \leq Ch^{-1} \sqrt{e^h},$ for all $p\in [1,\infty)$,
\end{itemize}
\end{lemma}
The proof follows from Lemma \ref{approx} given in Appendix A, in view of:
$$E(u^h, S^h) = \int_{S^h}\mathrm{dist }^2(\nabla u^h, SO(3)) 
\leq C \int_{S^h} W(\nabla u^h) \leq C h e^h$$
so that $\lim_{h\to 0} h^{-3} E(u^h, S^h) = 0$ by hypothesis {\bf (H)}.

\begin{lemma} \label{lem3.2}
Assume {\bf (H)} and let $R^h, Q^h$ be given as in Lemma \ref{appr}.
There holds:
\begin{itemize}
\item[(i)] $\displaystyle \lim_{h\to 0}~ (Q^{h})^TR^{h} =\mathrm{Id}$, in $W^{1,2}(S)$ and in
$L^p(S)$.
\end{itemize}
Moreover, there exists a $W^{1,2}$ skew-symmetric matrix fields $A:S\longrightarrow so(3)$ such that:
\begin{itemize}
\item[(ii)] $\displaystyle \lim_{h\to 0}~ \frac{h}{\sqrt{e^{h}}} \left( (Q^{h})^TR^{h} 
- \mathrm{Id}\right) = A$, weakly in $W^{1,2}(S)$ and (strongly) in $L^p(S)$.
\item[(iii)] $\displaystyle \lim_{h\to 0}~ \frac{h^2}{e^{h}} \mathrm{sym }\left( (Q^{h})^TR^{h} 
- \mathrm{Id}\right) = \frac{1}{2}A^2$, in $L^p(S)$.
\end{itemize}
In (ii) and (iii) convergence is up to a subsequence (that we do not relabel).
In (i), (ii), and (iii) the appropriate convergence holds for all $p\in [1,\infty)$.
\end{lemma}
\begin{proof}
The convergences in (i) follow from Lemma
\ref{appr} in view of {\bf (H)}.
To prove (ii), notice that the sequence:
$$A^h = \frac{h}{\sqrt{e^h}} \left((Q^h)^T R^h - \mathrm{Id}\right)$$
is bounded in $W^{1,2}(S)$ and so it has a weakly converging subsequence. By compact embedding of
$W^{1,2}(S)$ into $L^p(S)$ the convergence is strong in $L^p(S)$. One has:
$$A^h + (A^h)^T = \frac{h}{\sqrt{e^h}} \left((Q^h)^T R^h +  (R^h)^TQ^h - 2\mathrm{Id} \right) 
= - \frac{\sqrt{e^h}}{h}(A^h)^T\cdot A^h. $$
The latter converges to $0$ in $L^p(S)$, and therefore the limit matrix field $A$ is skew-symmetric.
The above equality proves as well that:
$$\lim_{h\to 0}  \frac{h}{\sqrt{e^h}} \mbox{ sym } A^h = \frac{1}{2} A^2$$
in $L^p(S)$, which implies (iii).
\end{proof}

Recall the rescaling:
\begin{equation*}
y^h(x + t\vec n(x)) = u^h\left(x + t{h}/{h_0}\vec n(x)\right) 
\qquad \forall x\in S \quad \forall t\in (-h_0/2, h_0/2),
\end{equation*}
so that $y^h\in W^{1,2}(S^{h_0}, \mathbb{R}^3)$.
Also, define:
\begin{equation*}
\nabla_h y^h(x + t\vec n(x)) = \nabla u^h\left(x + t{h}/{h_0}\vec n(x)\right)
\end{equation*}

By a straightforward calculation we obtain:
\begin{proposition}\label{formule}
For each $x\in S$, $t\in (-h_0/2, h_0/2)$ and $\tau\in T_xS$ there hold:
\begin{equation*}
\begin{split}
\partial_{\tau} y^h(x+t\vec n) &= 
\nabla_h y^h\left(x + t\vec n\right)\left(\mathrm{Id} + t{h}/{h_0}\Pi(x)\right)
(\mathrm{Id} + t\Pi(x))^{-1} \tau \\
{\partial_{\vec n}} y^h(x+t\vec n) &= \frac{h}{h_0} 
\nabla_h y^h\left(x + t\vec n\right)\vec n(x).
\end{split}
\end{equation*}
Moreover, for $I^h(y^h) = \frac{1}{h} \int_{S^h}W(\nabla u^h)$ one has:
\begin{equation*}
\begin{split}
I^h(y^h) &= \frac{1}{h_0}\int_{S^{h_0}} W(\nabla_h y^h (x+t\vec n))\cdot 
\det \left[\left(\mathrm{Id} + t{h}/{h_0}\Pi\right)
(\mathrm{Id} + t\Pi)^{-1}\right]\\
&= \int_S \fint_{-h_0/2}^{h_0/2} W(\nabla_h y^h (x+t\vec n))\cdot 
\det \left[\mathrm{Id} + t{h}/{h_0}\Pi(x)\right]~\mathrm{d}t~\mathrm{d}x.
\end{split}
\end{equation*}
\end{proposition}
Also, directly from Lemma \ref{appr} (i) and Lemma \ref{lem3.2} (ii) 
there follows:
\begin{proposition}\label{help}
Assume {\bf(H)}. Then:
\begin{itemize}
\item[(i)] $\displaystyle \|\nabla_h y^h - R^h\pi\|_{L^2(S^{h_0})}\leq C\sqrt{e^h}$. 
\item[(ii)] $\displaystyle \lim_{h\to 0} \frac{h}{\sqrt{e^h}}\left((Q^h)^T\nabla_h y^h - \mathrm{Id}\right)
= A\pi$, in $L^{2}(S^{h_0})$ up to a subsequence.
\end{itemize}
\end{proposition}

We will consider the corrected by rigid motions deformations 
$\tilde y^h\in W^{1,2}(S^{h_0},\mathbb{R}^3)$ and averaged displacements
$V^h\in W^{1,2}(S, \mathbb{R}^3)$:
\begin{equation*}
\tilde y^h = (Q^h)^T y^h - c^h,\qquad
V^h = V^h[\tilde y^h] 
= \frac{h}{\sqrt{e^h}}\fint_{-h_0/2}^{h_0/2} \tilde y^h(x+t\vec n) - x ~\mathrm{d}t,
\end{equation*}
where $c^h = \fint_S\fint_{-h_0/2}^{h_0/2} (Q^h)^T y^h - x ~\mbox{d}t~\mbox{d}x$,
so that $\fint_S V^h = 0$.

\begin{lemma}\label{lem3.4}
Assume {\bf(H)}. Then:
\begin{itemize}
\item[(i)] $\displaystyle \lim_{h\to 0} \tilde y^{h} = \pi,$ in $W^{1,2}(S^{h_0}).$
\item[(ii)] $\displaystyle \lim_{h\to 0} V^{h} = V,$  in $W^{1,2}(S)$ up to a subsequence.
\end{itemize}
The vector field $V$ in (ii) has regularity $W^{2,2}(S, \mathbb{R}^3)$ 
and it satisfies $\partial_\tau V (x) = A(x) \tau$
for all $\tau\in T_x S$. The $W^{1,2}$ skew-symmetric matrix field
$A:S\longrightarrow so(3)$ is as in Lemma \ref{lem3.2}.
\end{lemma}
\begin{proof} {\bf 1.} 
In view of Proposition \ref{formule} and Proposition \ref{help} we have:
\begin{equation}\label{important}
\begin{split}
&\left\|\nabla_{tan}\tilde y^h - \left((Q^h)^T R^h\right)_{tan}\cdot 
(\mathrm{Id} + th/h_0 \Pi) (\mathrm{Id} + t\Pi)^{-1} 
\right\|_{L^2(S^{h_0})} \leq C\sqrt{e^h}\\
&\left\|\partial_{\vec n}\tilde y^h \right\|_{L^2(S^{h_0})} \leq
Ch\|\nabla_h y^h\|_{L^2(S^{h_0})} \leq Ch.
\end{split}
\end{equation}
To prove convergence of $V^h$, consider:
\begin{equation}\label{nabla_Vh}
\begin{split}
\nabla V^h(x) &= \frac{h}{\sqrt{e^h}} \fint_{-h_0/2}^{h_0/2} \nabla_{tan}\tilde y^h(x+t\vec n)
(\mathrm{Id} + t\Pi) - \mathrm{Id} ~\mbox{d}t\\ 
&= \frac{h}{\sqrt{e^h}} \fint_{-h_0/2}^{h_0/2} \Big(\nabla_{tan}\tilde y^h
- \left((Q^h)^T R^h\right)_{tan} (\mathrm{Id} + t\Pi)^{-1}\Big) (\mathrm{Id} + t\Pi) ~\mbox{d}t\\ 
&\quad + \frac{h}{\sqrt{e^h}} \Big((Q^h)^T R^h (x) - \mathrm{Id}\Big)_{tan}.
\end{split}
\end{equation}
We see that by (\ref{important}) the first term in the right hand side above converges to $0$
in $L^2(S^{h_0})$, as $h\to 0$.  The second term converges, up to a subsequence, 
to $A_{tan}$ by Lemma \ref{lem3.2} (ii).
Therefore $\nabla V^h$  converges to $A_{tan}$ in $L^2(S)$ and since $\fint_S V^h = 0$, we may use
Poincar\'e inequality on $S$ to deduce (ii).

{\bf 2.}
To prove (i), notice that by (\ref{important}) and Lemma \ref{lem3.2} we obtain the following 
convergences in $L^2(S^{h_0})$:
\begin{equation*}
\begin{split}
&\lim_{h\to 0}\nabla_{tan}\tilde y^h= (\mathrm{Id} + t\Pi)^{-1} = \nabla_{tan} \pi,\\
&\lim_{h\to 0}\partial_{\vec n}\tilde y^h = 0.
\end{split}
\end{equation*}
Therefore $\nabla\tilde y^h$ converges to $\nabla\pi$ in $L^2(S^{h_0})$.

Since the sequence $\{V^h\}$ is bounded in $L^2(S)$, it also follows that:
\begin{equation}\label{conve}
\lim_{h\to 0} \left\|\int_{-h_0/2}^{h_0/2}\tilde y^h - \pi ~\mbox{d}t\right\|_{L^2(S)} = 0.
\end{equation}
Now, let $g(x+t\vec n) = |\mbox{det } (\mathrm{Id} +t\Pi(x))|^{-1}$. Consider
the two terms in the right hand side of:
$$\|\tilde y^h - \pi\|_{L^2(S^{h_0})} \leq
\left\|(\tilde y^h - \pi) - \int_{S^{h_0}}(\tilde y^h - \pi)\cdot \frac{g}{\scriptstyle \int_{S^{h_0}} g}
\right\|_{L^2(S^{h_0})}
+ ~\left|\int_{S^{h_0}}(\tilde y^h - \pi)\cdot \frac{g}{\scriptstyle\int_{S^{h_0}} g}\right|.$$
The first term  can be bounded by means of the weighted Poincar\'e inequality, by
$\|\nabla (\tilde y^h -\pi)\|_{L^2(S^{h_0})}$ and therefore it converges to $0$ as $h\to 0$.
The second term converges to $0$ as well, in view of (\ref{conve}) and:
$$\left|\int_{S^{h_0}} (\tilde y^h - \pi)\cdot g\right| = 
\left|\int_S\int_{-h_0/2}^{h_0/2} \tilde y^h - \pi ~\mbox{d}t~\mbox{d}x\right| \leq
C \left\|\int_{-h_0/2}^{h_0/2}\tilde y^h - \pi ~\mbox{d}t\right\|_{L^2(S)}.$$
This justifies convergence of $\tilde y^h$ to $\pi$ in $L^{2}(S^{h_0})$ 
and ends the proof of (i).
\end{proof}

Towards the proof of Theorem \ref{thmainuno}, we need 
to consider the following sequence of matrix fields on $S^{h_0}$:
$$G^h = \frac{1}{\sqrt{e^h}} \Big((R^h)^T \nabla_h y^h - \mathrm{Id}\Big).$$
In view of Proposition \ref{help} (i), $2\mbox{sym} G^h$ is the $\sqrt{e^h}$ order term in the 
expansion of the nonlinear strain $(\nabla u^h)^T \nabla u^h$, at $\mbox{Id}$.
This expression will also play a major role in the expansion 
of the energy density at $\mathrm{Id}$:
$W(\nabla_h y^h) = W(\mathrm{Id} + \sqrt{e^h} G^h)$.

\begin{lemma}\label{lem3.6}
Assume {\bf (H)}.  Then the sequence $\{G^h\}$ as above has a subsequence, converging weakly
in $L^2(S^{h_0})$ to a matrix field $G$.  The tangential minor of $G$ is, moreover,
affine in the $\vec n$ direction.  More precisely:
$$\forall \tau\in T_x S \qquad 
G(x+t\vec n)\tau = G_0(x)\tau + \frac{t}{h_0} \cdot \Big(\nabla (A\vec n)(x) - A\Pi (x)\Big)\tau,$$
where $\displaystyle G_0(x) = \fint_{-h_0/2}^{h_0/2} G(x + t\vec n)~\mathrm{d}t$.
\end{lemma}
\begin{proof}
{\bf 1.} The sequence $\{G^h\}$ is bounded in $L^2(S^{h_0})$ by Proposition \ref{help} (i).  
Therefore it has a subsequence (which we do not relabel)
converging weakly to some $G$.

For a fixed $s>0$, consider now the sequence of vector fields
$f^{s,h}\in W^{1,2}(S^{h_0}, \mathbb{R}^3)$:
$$f^{s,h}(x+t\vec n) = \frac{1}{s\sqrt{e^h}} \Big[\Big(h_0\tilde y^h(x+(t+s)\vec n) - h(x+(t+s)\vec n)\Big)
- \Big(h_0\tilde y^h(x+t\vec n) - h(x+t\vec n)\Big)\Big]$$
We claim that $f^{s,h}$ converges in $L^2(S^{h_0})$ (up to a subsequence) to $(A\vec n)\pi$ 
as $h\to 0$. Indeed, using Proposition \ref{formule} one has:
\begin{equation*}
\begin{split}
f^{s,h}(x+t\vec n) &= \frac{1}{\sqrt{e^h}} 
\fint_t^{t+s} \Big(h_0\partial_{\vec n} \tilde y^h(x+\sigma \vec n) - h\vec n\Big)
~\mbox{d}\sigma\\ &= \frac{h}{\sqrt{e^h}} \fint_{t}^{t+s} 
\Big((Q^h)^T\nabla_h y^h (x+\sigma\vec n) - \mathrm{Id}\Big)\vec n~\mbox{d}\sigma,
\end{split}
\end{equation*}
and the convergence follows by Proposition \ref{help} (ii).

\medskip

{\bf 2.}
We claim that this convergence is actually weak in $W^{1,2}(S^{h_0})$.
First, notice that the normal derivatives converge to $0$ in $L^2(S^{h_0})$ 
by Proposition \ref{help} (ii):
\begin{equation*}
\partial_{\vec n}f^{s,h}(x+t\vec n) = \frac{h}{s\sqrt{e^h}} ~
(Q^h)^T \Big(\nabla_h y^h(x+(t+s)\vec n) - \nabla_h y^h(x+t\vec n)\Big)\vec n(x).
\end{equation*}
We now find the weak limit of the tangential gradients of $f^{s,h}$.
By Proposition \ref{formule} there holds, for all $\tau\in T_x S$:
\begin{equation*}
\begin{split}
&\partial_{\tau}f^{s,h}(x+t\vec n) = \frac{1}{s\sqrt{e^h}} ~
\Big(h_0\nabla \tilde y^h(x+(t+s)\vec n)(\mathrm{Id} + (t+s)\Pi)(\mathrm{Id} + t\Pi)^{-1} \\
& \qquad\qquad\qquad\qquad\qquad\qquad\qquad\qquad\qquad
- h_0\nabla \tilde y^h(x+t\vec n) - hs\Pi (\mathrm{Id} + t\Pi)^{-1}\Big)\tau\\
&=\frac{h_0}{s\sqrt{e^h}} ~
(Q^h)^T \Big(\nabla_h y^h(x+(t+s)\vec n) - \nabla_h y^h(x+t\vec n)\Big)
(\mathrm{Id} + th/h_0\Pi)(\mathrm{Id} + t\Pi)^{-1}\tau \\
&\quad + \frac{h}{s\sqrt{e^h}}~\Big((Q^h)^T\nabla_h y^h(x+(t+s)\vec n) - \mathrm{Id}\Big)
s \Pi(\mathrm{Id} + t\Pi)^{-1}\tau.
\end{split}
\end{equation*}
By Proposition \ref{help} (ii), the second term in the right hand side above:
$$\frac{h}{\sqrt{e^h}}~\Big((Q^h)^T\nabla_h y^h(x+(t+s)\vec n) - \mathrm{Id}\Big)
\Pi(\mathrm{Id} + t\Pi)^{-1}$$
converges in $L^2(S^{h_0})$ to $A\Pi (\mathrm{Id} + t\Pi)^{-1}$.

On the other hand, the first term equals to:
$$\frac{h_0}{s} (Q^h)^T R^h \Big(G^h(x+(t+s)\vec n) - G^h(x+t\vec n)\Big)
(\mathrm{Id} + th/h_0\Pi)(\mathrm{Id} + t\Pi)^{-1}$$
and by Lemma \ref{lem3.2} (i) it converges weakly in $L^2(S^{h_0})$ to
$$\frac{h_0}{s} \Big(G(x+(t+s)\vec n) - G(x+t\vec n)\Big)
(\mathrm{Id} + t\Pi)^{-1}.$$
This establishes the (weak) convergence of $f^{s,h}$ in $W^{1,2}(S^{h_0})$.

\medskip

{\bf 3.} Equating the weak limits of tangential derivatives, we obtain, 
for every $\tau\in T_x S$:
\begin{equation*}
\begin{split}
\partial_\tau (A\vec n) (x) \cdot (\mathrm{Id} + t\Pi)^{-1} &=
\frac{h_0}{s}\Big(G(x+(t+s)\vec n) - G(x+t\vec n)\Big)
(\mathrm{Id} + t\Pi)^{-1}\tau \\
& \quad + A\Pi (\mathrm{Id} + t\Pi)^{-1}\tau.
\end{split}
\end{equation*}
This proves the lemma.
\end{proof}

Finally, we have the following bound for convergence of the scaled energies $I^h$:

\begin{lemma}\label{liminf}
Assume {\bf (H)}. Then:
$$\liminf_{h\to 0} \frac{1}{e^h} I^h(y^h) \geq \frac{1}{2} \int_S \mathcal{Q}_2\left(x,(\mathrm{sym }~ G_0)_{tan}\right)
+ \frac{1}{24} \int_S \mathcal{Q}_2\left(x,(\nabla(A\vec n) - A\Pi)_{tan}\right).$$
\end{lemma}
\begin{proof}
By the frame invariance property of $W$ we have:
\begin{equation*}
W(\nabla_h y^h) = W((R^h)^T \nabla_h y^h) = W(\mbox{Id} + \sqrt{e^h} G^h).
\end{equation*}
Consider the sets $\Omega_h=\{x\in S^{h_0}; ~~ (e^h)^{1/4} |G^h(x)|\leq 1\}$.
Clearly the sequence of characteristic functions $\chi_{\Omega_h}$ converges to
$1$ in $L^1(S^{h_0})$, as $\{(e^h)^{1/4} G^h\}$ converges pointwise to $0$.
Since $W$ is $\mathcal{C}^2$ in a neighborhood of $\mbox{Id}$, then by the
above calculation, in $\Omega_h$ (for $h$ sufficiently small) there holds:
\begin{equation}\label{modulus_continuity}
\begin{split}
\frac{1}{e^h}W(\nabla_h y^h) &= \frac{1}{e^h}\frac{1}{2} D^2 W(\mbox{Id})
(\sqrt{e^h} G^h, \sqrt{e^h} G^h) \\
&~~~~~~
+ \int_0^1 (1-s) \left[D^2W (\mbox{Id} + s\sqrt{e^h} G^h) - D^2W(\mbox{Id})
\right]~\mbox{d}s ~(G^h, G^h)\\
&= \frac{1}{2} \mathcal{Q}_3( G^h) + {o}(1) |G^h|^2.
\end{split}
\end{equation}
Above ${o}(1)$ is the Landau symbol denoting any quantity uniformly
converging to $0$, as $h\to 0$.
In view of Proposition \ref{formule} we now obtain:
\begin{equation*}
\begin{split}
\liminf_{h\to 0} \frac{1}{e^h}& I^h(y^h) \geq \liminf_{h\to 0} \frac{1}{e^h}\int_S\fint_{-h_0/2}^{h_0/2}
\chi_{\Omega_h} W(\nabla_h y^h) \mbox{det } [\mbox{Id} + th/h_0\Pi]~\mbox{d}t~\mbox{d}x\\
&= \liminf_{h\to 0}\int_S\fint_{-h_0/2}^{h_0/2}\chi_{\Omega_h}\frac{1}{e^h} W(\nabla_h y^h)~\mbox{d}t~\mbox{d}x\\
&= \liminf_{h\to 0}\frac{1}{2}\int_S\fint_{-h_0/2}^{h_0/2} \mathcal{Q}_3\left(\mathrm{sym }~ (\chi_{\Omega_h}G^h)\right)
+ {o}(1) \int_{S^{h_0}} |G^{h}|^2\\
&\geq \frac{1}{2}\int_S\fint_{-h_0/2}^{h_0/2}\mathcal{Q}_3(\mathrm{sym }~ G).
\end{split}
\end{equation*}
The last inequality follows from positive definiteness of $\mathcal{Q}_3$ on symmetric matrices,
and the fact that $\chi_{\Omega_h}G^h$ converges weakly to $G$, in $L^2(S^{h_0})$.

By the definition of $\mathcal{Q}_2$ and by Lemma \ref{lem3.6} we get:
\begin{equation*}
\begin{split}
&\frac{1}{2}\int_S\fint_{-h_0/2}^{h_0/2}\mathcal{Q}_3(\mathrm{sym }~ G)
= \frac{1}{2}\int_S\fint_{-h_0/2}^{h_0/2}\mathcal{Q}_2\left(x,(\mathrm{sym }~ G)_{tan}\right)\\
&\qquad
= \frac{1}{2} \Bigg[\int_S \fint_{-h_0/2}^{h_0/2} \mathcal{Q}_2\left(x,(\mathrm{sym }~ G_0)_{tan}\right)
+ \int_S \fint_{-h_0/2}^{h_0/2}\frac{t^2}{h_0^2}\mathcal{Q}_2\left(x,(\nabla(A\vec n) 
- A\Pi)_{tan}\right)\Bigg],
\end{split}
\end{equation*}
which proves the result.
\end{proof}

\section{A proof of Theorem \ref{thmainuno} and some explanations}
\label{sec_proofthuno}

To complete the proof of Theorem \ref{thmainuno}, in view of Lemma \ref{lem3.4}
and Lemma \ref{liminf}, it remains to
understand the
structure of the admissible matrices $G_0$.  This is the content of the
next lemma.

In addition to the hypothesis {\bf (H)}, we now also assume
the existence of the finite limit:
\begin{equation}\label{kappa}
\kappa = \lim_{h\to 0} \sqrt{e^h}/h^2 < \infty.
\end{equation}
When $e^h\approx h^\beta$, this corresponds to the case $\beta\geq 4$,
with $\kappa > 0$ for $\beta=4$ and $\kappa = 0$ for $\beta > 4$.

\begin{lemma}\label{lem3.9}
Assume {\bf (H)} and (\ref{kappa}).
Let $G_0$ be the matrix field on $S$, as in Lemma \ref{lem3.6}. 
Then we have the following convergence, up to a subsequence, weakly in $L^2(S)$:
\begin{equation}\label{symG0tan}
\lim_{h\to 0} \frac{1}{h} \mathrm{sym }~\nabla V^h 
= \left(\mathrm{sym }~G_0 + \frac{\kappa}{2} A^2\right)_{tan},
\end{equation}
where the subscript $_{tan}$ denotes, as usual, the tangential minor of 
a given matrix field on $S$.
\end{lemma}
\begin{proof}
We use the formula (\ref{nabla_Vh}) to calculate $\frac{1}{h}\mathrm{sym }~ 
\nabla V^h$.
The last term in the right hand side gives:
$$\frac{1}{\sqrt{e^h}} \mbox{sym} \left((Q^h)^T R^h - \mbox{Id}\right)_{tan}
= \frac{\sqrt{e^h}}{h^2} \frac{h^2}{e^h} \mbox{sym}\left((Q^h)^T R^h - \mbox{Id}\right)_{tan},$$
which converges in $L^2(S)$ to $\kappa/2 (A^2)_{tan}$ by Lemma \ref{lem3.2} (iii).

To treat the first term in the right hand side of (\ref{nabla_Vh}), notice that
for every $\tau\in T_xS$:
\begin{equation*}
\begin{split}
&\frac{1}{\sqrt{e^h}} \Bigg[\fint_{-h_0/2}^{h_0/2} \nabla\tilde y^h(x+t\vec n) (\mbox{Id} + t\Pi)
- (Q^h)^T R^h(x) ~\mbox{d}t\Bigg]\tau\\
&\quad = \frac{1}{\sqrt{e^h}}(Q^h)^T \Bigg[\fint_{-h_0/2}^{h_0/2}\nabla_h y^h(x+t\vec n) - R^h(x) 
~\mbox{d}t + \fint_{-h_0/2}^{h_0/2} t h/h_0 \nabla_h y^h \Pi~\mbox{d}t\Bigg] \tau\\
&\quad= \frac{1}{\sqrt{e^h}}(Q^h)^T R^h(x)
\Bigg[\fint_{-h_0/2}^{h_0/2}(R^h)^T\nabla_h y^h - \mbox{Id}~\mbox{d}t\Bigg]\tau \\
&\qquad\qquad\qquad\qquad
+ \frac{h/h_0}{\sqrt{e^h}} (Q^h)^T\Bigg[\fint_{-h_0/2}^{h_0/2} t \left(\nabla_h y^h- R^h\pi\right) 
~\mbox{d}t \Bigg]\Pi(x)\tau, 
\end{split}
\end{equation*}
where we used Proposition \ref{formule}.
Now, the second term in the right hand side above converges in $L^2(S)$ to $0$, 
by Proposition \ref{help} (i).
Further, the matrix in the first term equals to:
$$(Q^h)^T R^h(x) \fint_{-h_0/2}^{h_0/2}G^h(x+t\vec n)~\mbox{d}t,$$
and by Lemma \ref{lem3.2} (i) and Lemma \ref{lem3.6}, it converges weakly 
in $L^2(S)$ to $G_0$.
This completes the proof.
\end{proof}

We now comment on the regularity and role of various quantities containing 
$V$ and $A$, intrinsically related to the geometry of the problem.

\begin{remark}\label{rem3.6} 
Notice first that if a vector field $V\in W^{2,2}(S,\mathbb{R}^3)$ 
has skew-symmetric gradient:
\begin{equation*}
\tau\cdot \partial_\tau V(x) = 0 \qquad  \forall\tau\in T_x S,
\end{equation*} 
then it uniquely determines a $W^{1,2}$ matrix field $A:S\longrightarrow so(3)$ by:
\begin{equation}\label{diffic}
\begin{split}
\forall \tau\in T_xS \qquad A\tau &= \partial_\tau V,\\
A \vec n &= \Pi\cdot V_{tan} - \nabla_{tan} (V\vec n). 
\end{split}
\end{equation} 

Regarding the regularity, write $V$ as the sum of its tangential 
and normal components, to obtain:
$$V = V_{tan} + (V\vec n)\vec n,\qquad
\mbox{sym }\nabla V = \mbox{sym }\nabla V_{tan} + (V\vec n)\Pi.$$
Hence, assuming sufficient regularity of $S$ (say, $S$ is $\mathcal{C}^{3,1}$
up to its boundary) it follows that
$\mbox{sym }\nabla V_{tan} = -(V\vec n) \Pi \in W^{2,2}(S,\mathbb{R}^{2\times 2}).$
Using the same calculations as in \cite{ciarbookvol3} page 119, we may deduce
that the tangential component $V_{tan}$ enjoys higher regularity 
than the vector field $V$.  
Namely, $V_{tan}\in W^{3,2}(S,\mathbb{R}^3)$ and:
$$\|V_{tan}\|_{W^{3,2}(S)} \leq C\Big(\|V_{tan}\|_{W^{1,2}(S)}
+ \|V\vec n\|_{W^{2,2}(S)}\Big).$$
By Korn's inequality, one can replace the $W^{1,2}$ norm of $V_{tan}$ by
a term of the order $\|V_{tan}\|_{L^2} + \|\mbox{sym }\nabla V_{tan}\|_{L^2}$,
so that we finally obtain:
$$\|V_{tan}\|_{W^{3,2}(S)} \leq C\Big(\|V_{tan}\|_{L^{2}(S)}
+ \|V\vec n\|_{W^{2,2}(S)}\Big).$$
For an elementary derivation of Korn's inequality on $S$ from Korn's inequality 
on open sets, see e.g. \cite{LewMul}.

In the same manner, one can prove the following useful bound, valid under  
$\mathcal{C}^{2,1}$ regularity of $S$:
\begin{equation}\label{splitting}
\|V_{tan}\|_{W^{2,2}(S)} \leq C\Big(\|V_{tan}\|_{L^{2}(S)}
+ \|V\vec n\|_{W^{1,2}(S)}\Big).
\end{equation}

\end{remark}

\begin{remark}\label{rem4.3}
{\bf 1.}
The (scaled) $t$ - derivative of $G\tau$, which is also the argument of the
second term in the definition of $I$ (and $\tilde I$),  
may be written as:
$$\big(\nabla (A\vec n) - A\Pi \big)\tau 
= \Big[\big(\nabla (A\vec n) - A\Pi \big)\tau\Big]_{tan} 
= (\partial_\tau A)\vec n.$$
This expression measures the difference of order $h$ between the shape operator $\Pi$
on $S$ and the shape operator $\Pi^h$ of the deformed surface 
$S_h = (\mbox{id} + hV)(S)$ (see Figure \ref{figura1}).  
To see this, let $x\in S$ and let $\tau_1, \tau_2\in T_xS$ be such that $\vec n(x) =\tau_1\times \tau_2.$
The tangent map of the deformation $\phi^h(x) = x+hV(x)$ equals 
$\mbox{Id} + hA$, and we obtain the following expansion of the 
(scaled) normal vector $\vec n^h$ to $S_h$ at the point
$\phi^h(x)$:
\begin{equation*}
\begin{split}
\vec n^h &= \Big(\partial_{\tau_1}\phi^h \times \partial_{\tau_2}\phi^h\Big)(x) =
\vec n(x) + h(\tau_1\times \partial_{\tau_2}V + \partial_{\tau_1}V \times \tau_2)
+\mathcal{O}(h^2)\\
&= \vec n + hA\vec n + \mathcal{O}(h^2),
\end{split}
\end{equation*}
where we used the Jacobi identity for vector product and the fact that
$A\in so(3)$. Note that $|\vec n^h| = 1 + \mathcal{O}(h^2)$ and therefore
\begin{equation*}
\Pi^h (\mbox{Id} + hA) \tau = \partial_\tau \left (\frac{\vec n^h}{|\vec n^h|} \right )
= \partial_\tau \vec n^h + \mathcal{O}(h^2).  
\end{equation*}   
Hence the amount of bending of $S$, in the direction
of $\tau\in T_x S$, can be estimated by:
\begin{equation*}
\begin{split}
(\mbox{Id} + hA)^{-1} \Pi^h (\mbox{Id}& + hA)\tau - \Pi\tau
= (\mbox{Id} + hA)^{-1} (\partial_\tau \vec n^h + \mathcal{O}(h^2) )- \Pi\tau \\
& = (\mbox{Id} + hA)^{-1} \Big((\mbox{Id} + hA)\Pi\tau + h(\partial_\tau A)\vec n
+ \mathcal{O}(h^2)\Big) - \Pi\tau\\
& = (\mbox{Id} - hA) h (\partial_\tau A) \vec n + \mathcal{O}(h^2)
= h(\partial_\tau A)\vec n + \mathcal{O}(h^2).
\end{split}
\end{equation*}
\begin{figure}[h] 
\centerline{\psfig{figure=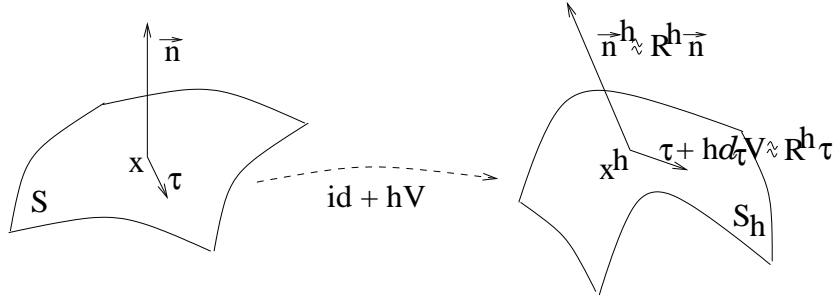,width=11cm,angle=0}}
\caption{The mid-surface $S$ and its deformation.}
\label{figura1}
\end{figure}
A closely related heuristics is the following. By Proposition \ref{help}
(for simplicity, we assume here that $e^h=h^4$)
the tangent map $\nabla u^h(x)$, at $x\in S$, is approximately a rotation
$R^h(x)\in SO(3)$. Hence, $\vec n^h \approx R^h \vec n$.
Assuming that $\lim Q^h = \mbox{Id}$, we may think that
$R^h(x) \approx \mbox{Id} + hA(x)$.
The difference of the shape operators on $S_h$ and $S$ satisfies:
\begin{equation*}
\begin{split}
 (R^h)^T \nabla \vec n^h -\Pi &\approx (\mbox{Id} + hA^T) \Big(\Pi + 
h\nabla_{tan}(A\vec n)\Big) - \Pi \\
&\approx h\nabla_{tan} (A\vec n) 
+ h A^T\Pi = h\Big(\nabla(A\vec n) - A\Pi\Big)_{tan}.
\end{split}
\end{equation*}

\medskip

{\bf 2.}
In turn, the role of the first term in the definition of $I$: 
$$(\mbox{sym }G_0)_{tan} = \lim_{h\to 0}
\frac{1}{h} \mbox{sym }\nabla V^h - \frac{\kappa}{2} (A^2)_{tan},$$
is to measure the difference of order $h^2$ between the metric on $S$ and the metric of
the deformed mid-surface. Notice that under the deformation $id + hV$, as in Figure \ref{figura1}, 
there is no first order change in the length of curves on $S$ 
because the gradient field $\nabla V$ is skew-symmetric. In geometrical terms, 
vector fields $V$ with this property 
are known as infinitesimal isometries (see \cite{Spivak}, Chapter 12).   

Under the same condition (for simplicity we again assume that $e^h=h^4$),
the amount of stretching of $S$, in the direction $\tau\in T_xS$ and
induced by the deformation $\phi^h = \mbox{id}+hV+ h^2 w$ has indeed 
the following expansion:
\begin{equation*}
\begin{split}
\left|\partial_{\tau} \phi^h\right|^2 -|\tau|^2
&= h^2\left(2\tau\partial_\tau w + |\partial_\tau V|^2\right) + \mathcal{O}(h^3)\\
&= 2h^2  \Big(\tau^T (\mathrm{sym }\nabla w)\tau - \frac{1}{2}\tau^T A^2\tau
\Big) + \mathcal{O}(h^3).
\end{split}
\end{equation*} 
\end{remark}

\section{The space of finite strains $\mathcal{B}$ and three examples 
of approximately robust surfaces}\label{sec_robust}

The space of limits as in the left hand side
of (\ref{symG0tan}) plays an important role in defining the exact limiting energy 
functional on $S$. With this in mind, we introduce:
\begin{definition}
The space of finite strains is the following closed subspace of $L^2(S)$:
$$\mathcal{B} = \Big\{\lim_{h\to 0}\mathrm{sym }~\nabla w^h; 
~~ w^h\in W^{1,2}(S,\mathbb{R}^3)\Big\}$$
where limits are taken in $L^2(S)$.
\end{definition}
Clearly, by Mazur's theorem, $\mathcal{B}$ contains all weak $L^2(S)$ 
limits of symmetric gradients of $W^{1,2}$ vector fields on $S$.

\medskip

As we shall see in Theorem \ref{thmaintre}, the form of the limiting energy 
functional simplifies, for surfaces with large space $\mathcal{B}$.

\begin{definition}\label{spaceB}
We say that $S$ is approximately robust, if for every $V\in W^{2,2}(S,\mathbb{R}^3)$
satisfying (\ref{Vass1}) and (\ref{Vass2}), one has: $(A^2)_{tan}\in\mathcal{B}$.
\end{definition}

According to our terminology, $S$ would be called ``robust'' if every admissible 
$(A^2)_{tan}$ as above, equaled $\mbox{sym }\nabla w$ for some 
$w\in W^{1,2}(S,\mathbb{R}^3)$.  The notion of robust surfaces will arise
at lower scalings, that is when $\kappa=\infty$, 
which we do not consider in this paper.

\begin{remark}
An equivalent construction of $\mathcal{B}$ is the following. 
Define the linear space of finite strain displacements:
$$\mathcal{W} = W^{1,2}(S,\mathbb{R}^3)/ \{w\in W^{1,2}; 
~~ \mathrm{sym }~\nabla w = 0\}.$$
It can be normed by $\|[w]\|_{\mathcal{W}} = \|\mathrm{sym }~\nabla w\|_{L^2(S)}$.
Then $(\mathcal{B}, \|\cdot\|_{L^2(S)})$ is linearly isometric to the completion
$\overline{\mathcal{W}}$ of $(\mathcal{W}, \|\cdot\|_{\mathcal{W}})$, 
so that the elements of $\mathcal{B}$ can be seen as generalized symmetric 
gradients of elements of $\overline{\mathcal{W}}$.

Such construction is used in \cite{GSP} or in \cite{ciarbookvol3} in
the context of derivation of membrane theories from linear elasticity. Note
the different regularity of the kernels considered in \cite{GSP} and the
related explanations in \cite{ciarbookvol3}, page 262.
\end{remark}

\begin{remark}
In general, it is complicated to directly 
determine the exact form of $\mathcal{B}$ or $\overline{\mathcal{W}}$.
The crucial step in identifying $\overline{\mathcal{W}}$ is finding 
the optimal norm $\|\cdot \|_o$ for which a Korn-Poincar\'e type inequality 
\begin{equation}\label{KP}
\inf\Big\{\|u- w\|_o;~~ w\in W^{1,2}, 
~\mathrm{sym }~\nabla w = 0\Big\} \leq C \|\mathrm{sym } \nabla u\|_{L^2(S)}
\end{equation} 
holds with a uniform constant $C$, for all $u\in W^{1,2}(S, \mathbb R^3)$. 
Unlike in the case of tangent vector fields, this optimal norm is usually weaker 
than $L^2$. The reason is that the boundedness of the left hand side in:
$$\mbox{sym }\nabla w^h = \mbox{sym }\nabla w^h_{tan} + (w^h\vec n) \Pi$$
does not, in general, imply $L^2$ boundedness of both terms in the right hand side.

This is the case, for example, when $S$ is (a piece of) a cylinder 
$\mathbb{S}^1\times [-1/2, 1/2]$. Let $\tau_1$ and $\tau_2$ be the tangent unit
vector fields, respectively orthogonal and parallel to the axis $x_3$ of the
cylinder.  One can show that there exists a sequence 
$[w^h] \in \mathcal{W}$ converging
in $\overline{\mathcal{W}}$, such that for any choice of representatives 
$w^h \in W^{1,2}(S,\mathbb R^3)$, the norms $\|w^h\tau_1\|_{L^2(S)}$
and $\|w^h\vec n \|_{H^{-1}(S)}$ blow up. However, this is the worst case
scenario, and one has: 
\begin{equation*}
\begin{split}
\overline{\mathcal{W}} =  \Big\{v \in\mathcal{D}'(S,\mathbb{R}^3); &~~
v\tau_1 \in H^{-1}(S), ~v\tau_2\in L^2(S),~ v\vec n \in H^{-2}(S),\\
&\mbox{sym }\nabla v \in L^2(S),\\  
&\int_{-1/2}^{1/2} x_3(v\vec n) \equiv const, 
~\int_{-1/2}^{1/2} v\vec n \equiv const, \\
&\int_S v\tau_1 = \int_S v\tau_2 = \int_S x_3(v\tau_1) =0 \Big\}.
\end{split}
\end{equation*}
In this particular case, however, as we will see below, $\overline{\mathcal{W}}$
is isometric to the space of all symmetric $L^2$ matrix fields $B_{tan}$ on $S$.
\end{remark}

\begin{remark}\label{flat}
Flat surfaces $S\subset \mathbb{R}^2$ are not approximately robust.
Indeed:
\begin{equation*}
\begin{split}
\mathcal{B} &= \left\{\mbox{sym }\nabla w;~~ w\in W^{1,2}(S,\mathbb{R}^3)\right\}\\
&= \Big\{B_{tan}\in L^2(S,\mathbb{R}^{2\times 2});~~ B_{tan}^T= B_{tan},
~ \mbox{curl}^T\mbox{curl } B_{tan} = 0\Big\}.
\end{split}
\end{equation*}
On the other hand, given $V\in W^{2,2}(S,\mathbb{R}^3)$
satisfying (\ref{Vass1}) and (\ref{Vass2}), one has $(A^2)_{tan}\in\mathcal{B}$
if and only if $V^3=V\vec n$ solves the degenerate Monge-Amp\`ere equation:
$\det \nabla^2(V^3)=0$ (see \cite{FJMhier}).
\end{remark}

\medskip

A particular class of surfaces which are approximately robust are surfaces for which:
\begin{equation}\label{wholeL2}
\mathcal{B} = \Big\{B_{tan}\in L^2(S,\mathbb{R}^{2\times 2}); ~~
B_{tan}^T= B_{tan}\Big\}.
\end{equation}
As we show below, three main examples of such surfaces are: 
convex surfaces, surfaces of revolution, and developable surfaces 
without flat regions.

\begin{lemma}\label{elptcapprobst}
Assume that $S$ is a simply connected, compact surface of class 
$\mathcal{C}^{2,1}$ with $\mathcal{C}^1$ boundary, and that its shape operator
$\Pi$ is strictly positive (or strictly negative) definite up to the boundary:
\begin{equation}\label{elliptic}
\forall x\in \bar S \quad \forall \tau\in T_xS \qquad
\frac{1}{C}|\tau|^2 \leq \Big(\Pi(x)\tau\Big)\cdot\tau \leq C|\tau|^2,
\end{equation}
Then $S$ is approximately robust, and more precisely (\ref{wholeL2}) holds.
\end{lemma}
\begin{proof} 
{\bf 1.} 
We will prove that every compactly supported, smooth symmetric
bilinear form $B_{tan}$ on $S$, must be of the form:
\begin{equation}\label{koc_uno}
B_{tan} = \mbox{sym }\nabla w,
\end{equation}
for some $w\in W^{1,2}(S, {\mathbb R}^3)$.
This will clearly imply the Lemma.
In \cite{Ni}, this result is proved under an additional assumption that $S$
is closed. The same method, with a slight modification, can be applied in our case.
For convenience of the reader, we present an overview of the argument
and for details of calculations we refer to \cite{Ni} and \cite{HH}, section 9.2.

Since $S$ is simply connected, it can be parametrized by a single chart 
$r\in  \mathcal{C}^{2,1} (\bar\Omega, \mathbb{R}^3)$, where 
$\Omega\subset {\mathbb R}^2$ is a simply connected domain with $\mathcal{C}^1$ 
boundary.  The definite form $[g_{ij}]_{i,j:1..2}$ with 
$g_{ij} = \partial_i r \cdot \partial_j r$ is the pull-back metric on $\Omega$ 
and $\sqrt{|g|}= \sqrt {\det[g_{ij}]}$ is the associated volume form. 
Also, the shape operator $\Pi$ expressed in the coordinates $(x_1, x_2)\in\Omega$ 
is given by $[h_{ij}]_{i,j:1..2}$, where
$h_{ij} = \partial_i (\vec n\circ r)\cdot\partial_j r.$ 
The inverse of $\Pi$ is  denoted $\Pi^{-1}= [h^{ij}]_{i,j:1..2}$. 
The mean curvature $H$ of $S$ equals to $\frac{1}{2}\mbox{tr }([g_{ij}]^{-1}\Pi)$.  

With the above notation, (\ref{koc_uno}) becomes
the following system of partial differential equations in $\Omega$: 
\begin{equation}\label{r-equation}
\left\{\begin{array}{l} 
\partial_1 r \cdot \partial_1 w = B_{11}  \\
\partial_1 r \cdot \partial_2 w + \partial_2 r \cdot \partial_1 w = 2 B_{12}\\
\partial_2 r \cdot \partial_2 w = B_{22}, 
\end{array} \right .
\end{equation} 
where we set $B_{ij}= \partial_i r \cdot B_{tan} \partial_j r$. 
Since $\mbox{sym }\nabla w$ is determined, one concentrates on the
skew part of $\nabla w$. Following \cite{Ni}, we let:
$$ \omega = \frac{1}{\sqrt{|g|}} \left(\partial_1 w \cdot\partial_2 r 
- \partial_2 w\cdot\partial_1 r\right), $$ 
and we observe that $\omega$ must satisfy the equation:
\begin{equation}\label{om-equation}
-\frac{1}{\sqrt{|g|}}\partial_i \left(\sqrt{|g|} h^{ij}\partial_j\omega\right) 
- 2 H \omega = {\mathcal D}(B_{ij}). 
\end{equation} 
The operator $\mathcal{D}:W^{2,2}(\Omega,\mathbb{R}^{2\times 2}) \longrightarrow 
L^2(\Omega,\mathbb{R})$  is a bounded differential operator
which depends on the geometry of $S$. The exact expression of $\mathcal{D}$ 
is given in the references mentioned before, but for our purposes it is enough 
to know its stated regularity.

Now, the following crucial relation between problems (\ref{om-equation})
and (\ref{r-equation}) is a direct consequence of calculations in \cite{Ni}.
\begin{proposition}\label{auxiliary-eq}
Assume that  $[B_{ij}]_{i,j:1..2}\in W^{2,2} (\Omega,\mathbb{R}^{2\times 2})$. 
If (\ref{om-equation}) has a (weak) solution $\omega\in  W^{1,2}(\Omega,\mathbb{R})$,
then the system (\ref{r-equation}) has a solution 
$w\in W^{1,2}(\Omega,\mathbb{R}^3)$. 
\end{proposition} 

\medskip

{\bf 2.} We now show that the hypothesis of Proposition \ref{auxiliary-eq}
is satisfied. Note that we have not imposed any boundary conditions on $\omega$, 
which makes the argument easier.
Extend first the coefficients $h^{ij}$ and $|g|$ to $\tilde h^{ij}$ and 
${|\tilde g|}$, respectively, defined on 
$\Omega_\epsilon = \{x\in {\mathbb R}^2;~~ \mbox{dist }(x,\Omega)<\epsilon\}$ 
for a small $\epsilon >0$. 
This extension can be made so that $[\tilde h^{ij}]_{i,j:1..2}$ satisfies 
the ellipticity condition (\ref{elliptic}) and that $\tilde h^{ij}$, 
$|\tilde g|$ and $1/|\tilde g|$ stay bounded in  $\Omega_\epsilon$. 

In order to prove existence of a solution to (\ref{om-equation}), we want
to find $f_0\in C^\infty_c(\Omega_\epsilon\setminus\Omega)$ such that 
the Dirichlet problem
\begin{equation}\label{om-equation2}
\mathcal{L}\omega = \mathcal{D}(B_{ij}) + f_0
\end{equation} 
has a solution $\omega\in W^{1,2}_0(\Omega_\epsilon,\mathbb{R})$.  
The restriction of $\omega$ to $\Omega$ will, clearly, serve our purpose.
Here the operator $\mathcal{L}$ is given: 
$$\mathcal{L}\omega =  - \frac{1}{\sqrt{|\tilde g|}}\partial_i\left(\sqrt{|\tilde g|}
\tilde h^{ij} \partial_j\omega\right) - 2 \tilde H\omega$$ 
is elliptic and self-adjoint with respect to the scalar product:
$$\langle \omega,\zeta\rangle = \int_{ \Omega_\epsilon}\omega\zeta\sqrt{|\tilde g|}.$$
Therefore, by the classical theory of elliptic operators 
(see e.g. \cite{Ev}, section 6.2, Theorem 4), (\ref{om-equation2}) has a solution 
if and only if its right hand side satisfies the orthogonality condition:
$$ \langle \mathcal{D}(B_{ij})+ f_0, \zeta\rangle = 0,$$
for all solutions $\zeta\in W^{1,2}_0(\Omega_\epsilon,\mathbb{R})$ of the
homogeneous problem: $\mathcal{L}\zeta = 0$ in $\Omega_\epsilon$.
The solution space of this problem is finite dimensional, say spanned by
a basis $\{\zeta_1, \ldots, \zeta_k\}$. 
For $f_0\in C^\infty_c(\Omega_\epsilon\setminus\bar\Omega)$ consider the functional:
$$ L(f_0) = \sum_{i=1}^k \langle f_0,\zeta_i\rangle e_i\in \mathbb{R}^k.$$ 
In view of the above, it suffices to prove that $L$ is surjective.

We now argue by contradiction.  Assume that there exists a nonzero 
$\alpha=(\alpha_1,\ldots, \alpha_k)\in\mathbb{R}^k$ orthogonal to the range of $L$.
In other words:
$$\int_{\Omega_\epsilon\setminus\Omega}\left(\sum_{i=1}^k\alpha_i\zeta_i\right)
f_0 \sqrt{|\tilde g|} = 0 
\quad \forall f_0\in C^\infty_c(\Omega_\epsilon\setminus\bar\Omega),$$
which clearly implies that $\sum_{i=1}^k \alpha_i\zeta_i = 0$ 
in $\Omega_\epsilon\setminus\bar\Omega$. 
By the H\"ormander uniqueness theorem for second order elliptic equations 
(see \cite{Ho}, Theorem 2.4), we obtain $\sum_{i=1}^k \alpha_i\zeta_i =0$ 
in $\Omega_\epsilon$, contradicting the linear independence of 
$\{\zeta_1,\ldots,\zeta_k\}$. 

In view of Proposition \ref{auxiliary-eq}, this ends the proof.
\end{proof}

\begin{lemma}\label{revolution}
Assume that $S$ is rotationally invariant, $\mathcal{C}^2$ up to the boundary,
and let $\bar{S}$ have no intersection with its axis of rotation.
Then (\ref{wholeL2}) holds.
\end{lemma}
\begin{proof}
{\bf 1.} After a suitable rigid motion, the surface $S$ can be parametrized by:
$$ r: (s_0, s_1) \times [0, 2\pi] \to {\mathbb R}^3, \qquad 
r(s,\theta) := g(s)\gamma(\theta) + s e_3, $$ 
for a positive function $g \in \mathcal{C}^2([s_0, s_1],\mathbb{R})$, $e_3 = (0,0,1)$, and 
$\gamma(\theta) = (\cos \theta, \sin \theta, 0)$. 

As in the proof of Lemma \ref{elptcapprobst}, we will show that (\ref{koc_uno}) has a solution
for $B_{tan}$ in an appropriate dense subset of the space in the right hand side 
of (\ref{wholeL2}).
Given $w\in W^{1,2}(S,\mathbb{R}^3)$, write:
$$ w(s,\theta) := a(s,\theta) \gamma(\theta) + b(s,\theta) \gamma'(\theta) + c(s,\theta) e_3 $$ 
and also let:
$$B_{ij} = \partial_i r \cdot B_{tan} \partial_j r.$$ 

The equation (\ref{koc_uno}) can now be expressed  as the following periodic system of partial 
differential equations in $(s_0, s_1) \times [0,2\pi]$ (see \cite{Spivak} chapter 12):
\begin{equation}\label{revolution-equation}
\left\{\begin{array}{l} 
g' \partial_s a + \partial_s c = B_{11} \\
\partial_{\theta} b + a = B_{22} \\
g' (\partial_{\theta} a - b) + g \partial_s b + \partial_{\theta} c = 2B_{12}.  
\end{array} \right.
\end{equation}  
We will prove that (\ref{revolution-equation}) has a solution $W^{1,2}$, periodic in 
$\theta\in [0,2\pi]$, for $B_{ij}$ being finite linear combinations of the Schauder basis 
for $L^2([s_0, s_1] \times [0, 2\pi])$ consisting of eigenfunctions of Laplacian 
under the periodic boundary conditions at $\theta\in\{0, 2\pi\}$ and Neumann boundary conditions 
in $s\in\{s_0, s_1\}$. By density, this will establish the lemma.

\medskip
 
{\bf 2.} Differentiating the third equation in $s$ and using the first two equations 
in (\ref{revolution-equation}) we obtain:
\begin{equation}\label{aux2}
g\partial^2_{s} b - g''(b + \partial^2_{\theta} b) = 2\partial_s B_{12} - \partial_\theta B_{11} 
- g''\partial_\theta B_{22} = :\psi(s,\theta).
\end{equation} 
Note that $\psi\in\mathcal{C}^0$ and for all $s$, $\psi (s,\cdot)$ is a finite 
linear combination of $\{e^{ik\theta}\}_{k\leq N}$, for some integer $N$ independent of $s$. 
Hence:
$$ b(s,\theta) = \sum_{-\infty}^{+\infty} b_k(s) e^{ik\theta} \quad \mbox{ and } \quad 
\psi(s,\theta) =  \sum_{-\infty}^{+\infty} \psi_k(s) e^{ik\theta},$$ 
with $\psi_k = \psi_{-k}$ and $\psi_k =0$ for $k> N$. 
Expressing (\ref{aux2}) in terms of the Fourier coefficients $b_k$ and $\psi_k$ we have:
$$ b''_k - \frac{g''}{g} ( 1- k^2)b_k = \frac{\psi_k}{g}.$$ 
Since the coefficients of the above linear equation are continuous in $[s_0, s_1]$, 
we deduce that there exist unique solutions
$b_k\in \mathcal{C}^2([s_0, s_1],\mathbb{R})$ satisfying $b_k(s_0) = b'_k(s_0) =0.$ 
Also, $b_k = b_{-k}$ and $b_k =0$ for $k>N$. Concluding, the finite linear combination
$b= \sum b_k(s)e^{ik\theta}$ is a $W^{2,2}$ solution to (\ref{aux2}), periodic in $\theta$.
One can now solve the first two equations in (\ref{revolution-equation}) for $a$ and then for $c$,
obtaining a $W^{1,2}$ solution to (\ref{revolution-equation}) and hence also to (\ref{koc_uno}).
\end{proof}

Finally, the following result has been proved in \cite{Schmidt}, Lemma 3.3:

\begin{lemma}
Let S be a $\mathcal{C}^2$ developable surface without flat regions.
That is, assume that for each $x\in S$ the Gauss curvature $\kappa(x) = 0$ while
$\Pi(x)\neq 0$. Then $S$ satisfies (\ref{wholeL2}).
\end{lemma}

\section{The recovery sequence: proofs of Theorems \ref{thmaindue} 
and \ref{thmaintre}}\label{sec_recovery}

In this section, we want to prove Theorems \ref{thmaindue} and \ref{thmaintre},
that is to define a suitable recovery sequence $y^h$.
Recall the definition (\ref{Qmatrices}).
With a slight abuse of notation, one can write:
\begin{equation}\label{marta}
\mathcal{Q}_2(x,F_{tan}) = \min\{\mathcal{Q}_3(F_{tan} + c\otimes \vec n(x) + \vec n(x)\otimes c); ~~
c\in\mathbb{R}^3\}.
\end{equation}
The unique vector $c$, for which the above minimum is attained
will be called $c(x,F_{tan})$. By uniqueness, the map $c$ is linear in its 
second argument.   

\medskip

Given $B_{tan}\in\mathcal{B}$, there exists a sequence of vector fields $w^h\in W^{1,2}(S,\mathbb{R}^3)$
such that $\mbox{sym }\nabla w^h$ converge in $L^2(S)$ to $B_{tan}$. Without loss of generality,
we may assume that $w^h$ are smooth, and
(by possibly reparametrizing the sequence) that:
\begin{equation}\label{norm}
\lim_{h\to 0} \sqrt{h}\|w^h\|_{W^{2,\infty}(S)} = 0.
\end{equation}
Let $V\in W^{2,2}(S,\mathbb{R}^3)$ be such that $\partial_\tau V(x) =A(x)\tau$,
for all $\tau\in T_xS$, where $A\in W^{1,2}(S,\mathbb{R}^{3\times 3})$ is a skew-symmetric matrix field.
We approximate $V$ by a sequence $v^h\in W^{2,\infty}(S,\mathbb{R}^3)$ such that, 
for a sufficiently small, fixed $\epsilon_0>0$:
\begin{equation}\label{vhapprox}
\begin{split}
& \lim_{h\to 0} \|v^h - V\|_{W^{2,2}(S)} = 0, 
\qquad \frac{\sqrt{e^h}}{h}\|v^h\|_{W^{2,\infty}(S)} \leq \epsilon_0,\\
&  \lim_{h\to 0}\frac{h^2}{e^h}~ \mu\left\{x\in S; ~~ v^h(x) \neq V(x)\right\} =0.
\end{split}
\end{equation}
The existence of such $v^h$ follows by partition of unity and a truncation argument, as a special case
of the Lusin-type result for Sobolev functions in \cite{Liu50} 
(see also Proposition 2 in \cite{FJMhier}).

\medskip

Define a sequence of rescaled deformations $y^h\in W^{1,2}(S^{h_0},\mathbb{R}^3)$:
\begin{equation}\label{rec_seq}
\begin{split}
y^h(x+t\vec n) & = x + \frac{\sqrt{e^h}}{h} v^h(x) + \sqrt{e^h} w^h(x) \\
& \qquad + {t}h/h_0\vec n(x) + {t}/h_0\sqrt{e^h} \Big(\Pi v^h_{tan} - \nabla (v^h\vec n)\Big)(x)\\
& \qquad - th/h_0 \sqrt{e^h} \vec n^T \nabla w^h
+ {t}h/h_0 \sqrt{e^h} d^{0,h}(x) + \frac{t^2}{2h_0^2}h\sqrt{e^h} d^{1,h}(x).
\end{split}
\end{equation}
Notice that if $V\in W^{2,\infty}(S)$ then one may take $v^h = V$ in which case the term
${t}/h_0\sqrt{e^h} (\Pi v^h_{tan} - \nabla (v^h\vec n))$ is exactly
${t}/h_0 \sqrt{e^h}A\vec n$ (see (\ref{diffic}) in Remark \ref{rem3.6}).

The vector fields $d^{0,h}, d^{1,h}\in W^{1,\infty}(S,\mathbb{R}^3)$ are defined so that:
\begin{equation}\label{nd01h}
\lim_{h\to 0} \sqrt{h} \left(\|d^{0,h}\|_{W^{1,\infty}(S)} + \|d^{1,h}\|_{W^{1,\infty}(S)}\right) = 0
\end{equation}
and:
\begin{equation}\label{warp}
\begin{split}
\lim_{h\to 0} d^{0,h} & = 2c\left(x, B_{tan} - \frac{\kappa}{2}(A^2)_{tan}\right) + {\kappa}A^2\vec n 
- \frac{1}{2} \kappa (\vec n^T A^2 \vec n)\vec n
\quad \mbox{ in } L^2(S),\\
\lim_{h\to 0} d^{1,h} & = 2c\left(x, \mbox{sym }(\nabla(A\vec n) - A\Pi)_{tan}\right)
+ \left(\vec n^T A\Pi - \vec n^T\nabla(A\vec n)\right)
\mbox{ in } L^2(S).
\end{split}
\end{equation}

\begin{lemma}\label{lemeasy}
Assume (\ref{kappa}). 
For the sequence $\{y^h\}$ in (\ref{rec_seq}) the convergences (i), (ii) and (iii)
of Theorem \ref{thmaindue} hold. 
\end{lemma}
\begin{proof}
(i) follows by the normalization (\ref{norm}), (\ref{vhapprox}) and (\ref{nd01h}).
For (ii) and (iii) notice that:
\begin{equation*}
\begin{split}
V^h[y^h] &= v^h + h w^h + \frac{1}{24} h^2 d^{1,h}\\
\frac{1}{h} \mbox{sym }\nabla V^h[y^h] & = \frac{1}{h}\mbox{sym } \nabla v^h +
\mbox{sym } \nabla w^h + \frac{1}{24} h \mbox{sym }\nabla d^{1,h}.
\end{split}
\end{equation*}
The proof will be achieved once we establish that:
\begin{equation}\label{bad_term}
\lim_{h\to 0} \frac{1}{h}\|\mbox{sym }\nabla v^h\|_{L^2(S)} = 0.
\end{equation}
Since the Lipschitz constant of each $\nabla v^h$ is bounded by $\epsilon_0 \frac{h}{\sqrt{e^h}}$,
and $\mbox{sym }\nabla v^h=0$ on the set $\{x\in S;~ v^h(x)= V(x)\}$, we have:
$$|\mbox{sym }\nabla v^h(x)|\leq C\frac{h}{\sqrt{e^h}} \mbox{dist } \Big(x, \{v^h=V\}\Big).$$
Now we claim that the right hand side above converges to $0$, in $L^\infty(S)$.
For otherwise there would be $\mbox{dist } (x^h, \{v^h=V\})\geq C\frac{\sqrt{e^h}}{h}$, for some
sequence $x^h\in S$. Consequently, denoting by $B_{x^h}(r)$ the ball in $\mathbb{R}^3$ centered
at $x^h$ and radius $r$, we would obtain:
$$\mu\{x\in S; ~ v^h(x)\neq V(x)\}\geq \left|S \cap B_{x^h}\left(\frac{1}{2}
\mbox{dist }(x^h,~ \{v^h= V\})\right)\right| \geq C\frac{e^h}{h^2}, $$
contradicting (\ref{vhapprox}). In the last inequality above we used that the surface $S$ is of class 
$\mathcal{C}^1$, with Lipschitz continuous boundary.
We thus obtain:
$$\lim_{h\to 0} \|\mbox{sym }\nabla v^h\|_{L^\infty(S)} = 0.$$
On the other hand:
\begin{equation*}
\begin{split}
\frac{1}{h}\|\mbox{sym }\nabla v^h\|_{L^2(S)} &\leq \frac{1}{h}
\mu\{x\in S; ~ v^h(x) \neq V(x)\}^{1/2} \cdot \|\mbox{sym } \nabla v^h\|_{L^\infty(S)}\\
&\leq C\frac{\sqrt{e^h}}{h^2} \|\mbox{sym } \nabla v^h\|_{L^\infty(S)}.
\end{split}
\end{equation*}
The two statements above imply (\ref{bad_term}).
\end{proof}

\bigskip

\noindent {\bf Proof of Theorem \ref{thmaindue}}.

\noindent We will prove that:
\begin{equation}\label{limsup}
\limsup_{h\to 0} \frac{1}{e^h} I^h(y^h) \leq I(V,B_{tan}) + \eta,
\end{equation}
where $\eta$ denotes an error quantity, with the property:
\begin{equation}\label{etaepsilon}
\eta\to 0 \quad \mbox{ as } \epsilon_0\to 0.
\end{equation}  
In view of Theorem \ref{thmainuno}, this will imply (iv) for a recovery sequence 
obtained through a diagonal argument, when $\epsilon_0\to 0$. Clearly, the assertions 
(i) - (iii) will follow as well, by Lemma \ref{lemeasy}.

\medskip

{\bf 1.}  We first look closer at quantities $\nabla_h y^h$.
By Proposition \ref{formule}, it follows that:
\begin{equation}\label{form}
\begin{split}
&(\nabla_h y^h) (x + t\vec n) \vec n(x) 
= \vec n + \frac{\sqrt{e^h}}{h} \Big(\Pi v_{tan}^h - \nabla(v^h\vec n)\Big) \\
& \qquad\qquad\qquad\qquad\qquad - \sqrt{e^h} \vec n^T\nabla w^h
+ \sqrt{e^h} d^{0,h} + t/h_0\sqrt{e^h} d^{1,h},\\
&(\nabla_h y^h) (x + t\vec n)\tau 
= \nabla y^h(x + t\vec n) (\mbox{Id} + t\Pi) (\mbox{Id} + th/h_0\Pi)^{-1}\tau  \\
& \qquad
= \Bigg(\mbox{Id} + \frac{\sqrt{e^h}}{h}\nabla v^h + \sqrt{e^h}\nabla w^h + th/h_0 \Pi\\
& \qquad\qquad
+ t/h_0 \sqrt{e^h}\nabla \Big(\Pi v_{tan}^h - \nabla(v^h\vec n)\Big)
- th/h_0 \sqrt{e^h}\nabla (\vec n^T \nabla w^h) \\
& \qquad\qquad\qquad
+ th/h_0 \sqrt{e^h} \nabla d^{0,h}
+ \frac{t^2}{2h_0^2}h \sqrt{e^h}\nabla d^{1,h}\Bigg)(\mbox{Id} + th/h_0\Pi)^{-1}\tau,
\end{split}
\end{equation}
for all $\tau\in T_xS$.

By (\ref{norm}), (\ref{vhapprox}) and (\ref{nd01h}) one has:
$\|\nabla_h y^h - \mbox{Id}\|_{L^\infty(S^{h_0})} \leq C\epsilon_0$.
It now follows by polar decomposition theorem (assuming $\epsilon_0$ to be sufficiently small),
that $\nabla_h y^h$ is a product of a proper rotation and the well defined square root 
of $(\nabla_h y^h)^T\nabla_h y^h$.
By properties of the energy density function and Taylor expansion, we obtain:
\begin{equation*}
W(\nabla_h y^h) = W\left(\sqrt{(\nabla_h y^h)^T\nabla_h y^h}\right)
= W\left(\mbox{Id} + \frac{1}{2} K^h + \mathcal{O}(|K^h|^2)\right),
\end{equation*}
where:
$$K^h =  (\nabla_h y^h)^T\nabla_h y^h - \mbox{Id}.$$
Clearly:
\begin{equation}\label{Khsmall}
\|K^h\|_{L^\infty(S^{h_0})} \leq C\epsilon_0,
\end{equation}
and so reasoning as in (\ref{modulus_continuity}), the above expansion in $W$ yields:
\begin{equation}\label{Wdopp}
\begin{split}
\frac{1}{e^h} W(\nabla_h y^h) = \frac{1}{2} \mathcal{Q}_3\left(\frac{1}{2\sqrt{e^h}} K^h + 
\frac{1}{\sqrt{e^h}}\mathcal{O}(|K^h|^2)\right) + 
\frac{1}{\sqrt{e^h}} \eta \cdot \mathcal{O}(|K^h|^2),
\end{split}
\end{equation}
where $\eta$ depends only on $\epsilon_0$ and satisfies (\ref{etaepsilon}).

\medskip

{\bf 2.} Using (\ref{form}) we now calculate $K^h$. By $Error$ we will 
cumulatively denote all the terms with the property:
\begin{equation}\label{def03}
\lim_{h\to 0}\frac{1}{\sqrt{e^h}} \|Error\|_{L^2(S^{h_0})} =0.
\end{equation}
We start with the tangential minor of $K^h$:
\begin{equation*}
\begin{split}
K^h_{tan}&(x + t\vec n)  = (\mbox{Id} + th/h_0\Pi)^{-1}\Bigg[\mbox{Id}
+ 2\frac{\sqrt{e^h}}{h}\mbox{sym } \nabla v^h + 2 \sqrt{e^h}\mbox{sym } \nabla w^h 
+ 2 th/h_0 \Pi \\
& \qquad
+ 2 t/h_0\sqrt{e^h}\mbox{sym } \nabla\Big(\Pi v_{tan}^h - \nabla(v^h\vec n)\Big)
+ \frac{e^h}{h^2} (\nabla v^h)^T\nabla v^h + t^2 h^2/h_0^2 \Pi^2\\
& \qquad
+ 2t \sqrt{e^h}/h_0 \mbox{sym }\Big(\Pi\nabla v^h\Big) 
+ Error \Bigg](\mbox{Id} + th/h_0\Pi)^{-1} - \mbox{Id}\\
& = (\mbox{Id} + th/h_0\Pi)^{-1}\sqrt{e^h}\Bigg[2 \mbox{sym } \nabla w^h 
+ \frac{\sqrt{e^h}}{h^2} (\nabla v^h)^T\nabla v^h\\
& \qquad\qquad\qquad\qquad \qquad\qquad
+ 2t/h_0\mbox{sym }\nabla\Big(\Pi v^h_{tan} - \nabla(v^h\vec n)\Big)\\
& \qquad\qquad \qquad\qquad\qquad\qquad
+ 2t/h_0\mbox{sym }\Big(\Pi\nabla v^h\Big)
\Bigg](\mbox{Id} + th/h_0\Pi)^{-1} + Error,
\end{split}
\end{equation*}
where we used the formulae:
\begin{equation*}
\begin{split}
&(\mbox{Id } + F)^T (\mbox{Id } + F) = \mbox{Id } + 2\mbox{sym } F + F^T F,\\
&F_1^{-1} F F_1^{-1} - \mbox{Id } = F_1^{-1} (F - F_1^{2}) F_1^{-1}.
\end{split}
\end{equation*}
Notice that the quantity $Error$ contains the term $\frac{\sqrt {e^h}}{h} \mbox{sym }\nabla v^h$,
resulting from the relaxation of the constraint (\ref{Vass1}), (\ref{Vass2}) 
on the small set $\{v^h\neq V\}$, and other product terms, eg:
$\frac{e^h}{h}(\nabla v^h)^T \nabla(\Pi v^h_{tan} - \nabla(v^h\vec n))$.
The convergence of $\frac{1}{h}\|\mbox{sym }\nabla v^h\|_{L^2(S^{h_0})}$ to $0$
has been proved in (\ref{bad_term}). All other terms in $Error$ can be dealt with by 
repeated use of (\ref{vhapprox}), (\ref{norm}), H\"older and Sobolev inequalities, eg:
\begin{equation*}
\begin{split}
\frac{\sqrt{e^h}}{h}\|(\nabla v^h)^T& \nabla(\Pi v^h_{tan} - \nabla(v^h\vec n))\|_{L^2(S)}
\leq C \frac{\sqrt{e^h}}{h} \|\nabla v^h\|_{L^4(S)} \|v^h\|_{W^{2,4}(S)} \\
&\leq  C \frac{\sqrt{e^h}}{h} \|\nabla v^h\|_{W^{1,2}(S)} \|v^h\|_{W^{2,\infty}(S)}^{1/2}
\|v^h\|_{W^{2,2}(S)}^{1/2} \\
&\leq  C \frac{\sqrt{e^h}}{h} \|v^h\|_{W^{2,\infty}(S)}^{1/2}
\longrightarrow 0 \quad \mbox{ as } h\to 0.
\end{split}
\end{equation*}

Now, the normal minor of $K^h$ is calculated as:
\begin{equation*}
\begin{split}
\vec n^T K^h(x+t\vec n)\vec n = \sqrt{e^h}\Bigg(\frac{\sqrt{e^h}}{h^2} 
\Big|\Pi v_{tan}^h - \nabla(v^h\vec n)\Big|^2 + 2 d^{0,h} \vec n + 2 t/h_0 d^{1,h} \vec n
\Bigg) + Error.
\end{split}
\end{equation*}

The remaining coefficients of the symmetric matrix $K^h(x+ t\vec n)$ are, for $\tau\in T_x S$:
\begin{equation*}
\begin{split}
 \tau^T &K^h(x+t\vec n)\vec n = \vec n^T\Bigg[\frac{\sqrt{e^h}}{h}\nabla v^h +
t/h_0 \sqrt{e^h} \nabla \Big(\Pi v^h_{tan} - \nabla(v^h\vec n)\Big)
\Bigg](\mbox{Id} + th/h_0\Pi)^{-1}\tau\\
& \qquad  
+ \Bigg[\frac{\sqrt{e^h}}{h}\Big(\Pi v^h_{tan} - \nabla(v^h\vec n)\Big)^T 
+ (\sqrt{e^h} d^{0,h} + t/h_0 \sqrt{e^h}d^{1,h})^T \\
&\qquad\qquad\qquad
+ \frac{e^h}{h^2}\Big(\Pi v^h_{tan} - \nabla(v^h\vec n)\Big)^T \nabla v^h \\
&\qquad\qquad\qquad
+ t/h_0 \sqrt{e^h}\Big(\Pi v^h_{tan} - \nabla(v^h\vec n)\Big)^T \Pi
\Bigg](\mbox{Id} + th/h_0\Pi)^{-1}\tau + Error\\
& = \sqrt{e^h} \Bigg[t/h_0 \vec n^T\nabla\Big(\Pi v^h_{tan} - \nabla(v^h\vec n)\Big) 
+ \frac{\sqrt{e^h}}{h^2} \Big(\Pi v^h_{tan} - \nabla(v^h\vec n)\Big) ^T \nabla v^h\\
& \qquad\qquad\qquad
+ t/h_0 \Big(\Pi v^h_{tan} - \nabla(v^h\vec n)\Big)^T\Pi \\
& \qquad\qquad\qquad\qquad\qquad\quad
+ (d^{0,h} + t/h_0 d^{1,h})^T 
\Bigg](\mbox{Id} + th/h_0\Pi)^{-1}\tau + Error.
\end{split}
\end{equation*}
We leave the estimation in $Error$ to the reader.
The convergence of the most troublesome term:
\begin{equation*}
\lim_{h\to 0} \frac{1}{h}\|\vec n^T \nabla v^h 
+ (\Pi v_{tan}^h - \nabla(v^h\vec n))^T\|_{L^2(S^{h_0})} = 0.
\end{equation*}
can be proved as in (\ref{bad_term}), since the quantity in question vanishes on the 
set $\{v^h = V\}$. Therefore, $\|\vec n^T \nabla v^h 
+ (\Pi v_{tan}^h - \nabla(v^h\vec n))^T\|_{L^\infty(S)}$ converges to $0$, as $h\to 0$,
and the displayed convergence follows by the last assertion in (\ref{vhapprox}).

\medskip

{\bf 3.} In view of (\ref{def03}) we may now write (with a slight abuse of notation)
\begin{equation}\label{Kconv}
\lim_{h\to 0} \frac{1}{2\sqrt{e^h}} K^h = K_1(x)_{tan} + \frac{t}{h_0}K_2(x)_{tan}
+ (\zeta\otimes \vec n + \vec n\otimes \zeta) \quad \mbox{ in } L^2(S^{h_0}),
\end{equation}
where the symmetric symmetric matrix fields $(K_i)_{tan}\in L^2(S, \mathbb{R}^{2\times 2})$
and the vector field $\zeta\in L^2(S^{h_0}, \mathbb{R}^3)$ are given by:
\begin{equation}\label{Kdef}
\begin{split}
K_1(x)_{tan} &= B_{tan} - \frac{\kappa}{2} (A^2)_{tan},\\
K_2(x)_{tan} &= \mbox{sym } \left(\nabla(A\vec n) - A\Pi\right)_{tan},\\
\zeta(x+t\vec n) & = c\Big(x,  B_{tan} - \frac{\kappa}{2} (A^2)_{tan}\Big)
+ \frac{t}{h_0} c\Big(x,\mbox{sym } \left(\nabla(A\vec n) - A\Pi\right)_{tan}\Big).
\end{split}
\end{equation}

Further, we observe:
\begin{equation}\label{Kthird}
\lim_{h\to 0} \frac{1}{e^h}\int_{S^{h_0}} |K^h|^4 = 0.
\end{equation}
Indeed, (\ref{Kconv}) implies that $\frac{1}{\sqrt{e^h}} K^h$ converges pointwise a.e. in $S^{h_0}$.
Thus $\frac{1}{e^h} |K^h|^4$ converges a.e. to $0$. By the boundedness of $K^h$ in (\ref{Khsmall}):
$\frac{1}{e^h}|K^h|^4 \leq C \frac{1}{e^h}|K^h|^2$, and the dominated convergence theorem 
achieves (\ref{Kthird}).

\medskip

{\bf 4.}
Finally, we prove now (\ref{limsup}). By (\ref{Kthird}), it follows that the argument 
of $\mathcal{Q}_3$ in (\ref{Wdopp}) converges in $L^2(S^{h_0})$ to the same limit as 
$\frac{1}{2\sqrt{e^h}} K^h$ in (\ref{Kconv}).
Using Proposition \ref{formule}, (\ref{Kconv}) and (\ref{marta}), we obtain:
\begin{equation*}
\begin{split}
\limsup_{h\to 0} \frac{1}{e^h}& I^h(y^h)  = 
\limsup_{h\to 0} \frac{1}{e^h} \int_S \fint_{-h_0/2}^{h_0/2} W(\nabla_h y^h) \cdot
\det (\mbox{Id} + th/h_0\Pi)~\mbox{d}t \mbox{d}tx\\
& \leq \frac{1}{2} \limsup_{h\to 0} \int_S \fint_{-h_0/2}^{h_0/2} \mathcal{Q}_3 \left(\frac{1}{2\sqrt{e^h}} 
K^h(x+t\vec n)\right)\cdot\det (\mbox{Id} + th/h_0\Pi)~\mbox{d}t\mbox{d}x \\
&\qquad\qquad\qquad
+ C\eta \limsup_{h\to 0}\frac{1}{e^h}\int_{S^{h_0}}|K^h|^2 \\
& = \frac{1}{2}\int_S \fint_{-h_0/2}^{h_0/2}\mathcal{Q}_3 \left(\lim_{h\to 0}\frac{1}{2\sqrt{e^h}}K^h\right)
+C\eta\left\|\lim_{h\to 0}\frac{1}{2\sqrt{e^h}}K^h\right\|^2_{L^2(S^{h_0})}\\
& \leq \frac{1}{2}\int_S \fint_{-h_0/2}^{h_0/2}\mathcal{Q}_2\left(x, K_1(x)_{tan}
+\frac{t}{h_0}K_2(x)_{tan}\right) ~\mbox{d}t\mbox{d}x  + C\eta\\
& = \frac{1}{2}\int_S \fint_{-h_0/2}^{h_0/2}\mathcal{Q}_2\left(x, K_1(x)_{tan}\right) 
+\frac{t^2}{h_0^2}\mathcal{Q}_2\left(x,K_2(x)_{tan}\right)~\mbox{d}t\mbox{d}x + C\eta,
\end{split}
\end{equation*}
which implies (\ref{limsup}) in view of (\ref{Kdef}).
\endproof

\medskip

\begin{remark}
A more careful calculation reveals the exact convergence:
$$\lim_{h\to 0}\frac{1}{e^h} I^h(y^h) = I(V, B_{tan}),$$
for the recovery sequence (\ref{rec_seq}).
We have used another argument for the sake of a more transparent presentation.
\end{remark}

\medskip

\noindent {\bf Proof of Theorem \ref{thmaintre}.}

\noindent When $\kappa = 0$, the recovery sequence (for a given $V\in W^{2,2}(S,\mathbb{R}^3)$ 
satisfying (\ref{Vass1}) and (\ref{Vass2})) is given again by (\ref{rec_seq}), where we put
$w^h = 0$, $B_{tan}=0$ and $\kappa = 0$.  That is:
\begin{equation*}
\begin{split}
y^h(x+t\vec n)  & = x + \frac{\sqrt{e^h}}{h} v^h(x)+ {t}h/h_0\vec n(x) \\
& \qquad + {t}/h_0\sqrt{e^h} \Big(\Pi v^h_{tan} - \nabla_{tan}(v^h\vec n)\Big)(x)
+ \frac{t^2}{2h_0^2}h\sqrt{e^h} d^{1,h}(x),
\end{split}
\end{equation*}
where $d^{1,h}\in W^{1,\infty}(S,\mathbb{R}^3)$ satisfies (\ref{nd01h}) and the second
formula in (\ref{warp}).

Clearly, $y^h$ and $V^h[y^h]$ converge in $W^{1,2}(S^{h_0})$ to $\pi$ and $V$, respectively,
as in Lemma \ref{lemeasy}. The convergence of the scaled energy follows as in 
Theorem \ref{thmaindue} (iv).
\endproof

\section{The convergence of minimizers:  proof of Theorem \ref{thmaincinque}}\label{sec_Gamma}

Recall that the considered sequence of forces $f^h\in L^2(S^h,\mathbb{R}^3)$
with zero mean: $\int_{S^h} f^h=0$, has the form: 
$$f^h(x+ t\vec n(x)) = h\sqrt{e^h}\det(\mbox{Id} + t\Pi(x))^{-1} f(x).$$

\begin{lemma}\label{Gammafh}
Let $u^h\in W^{1,2}(S^h,\mathbb{R}^3)$ be a sequence
of deformations such that $V^h[y^h]$ converges in $L^2(S)$ to some
$V:S\longrightarrow \mathbb{R}^3$ and let $Q^h\in\mathbb{R}^{3\times 3}$
converge to some $Q$.
Then:
\begin{equation*}
\lim_{h\to 0} \frac{1}{e^h} \frac{1}{h} \int_{S^h} f^h \cdot Q^h(u^h-{\rm id}) 
= \int_S f\cdot QV. 
\end{equation*}
\end{lemma}
\begin{proof}
We have:
\begin{equation*}
\begin{split}
\frac{1}{e^h} \frac{1}{h} \int_{S^h} f^h\cdot Q^h &(u^h -{\rm id})  
 = \frac{1}{e^h}\int_S f^h(x)\cdot Q^h \fint_{-h/2}^{h/2} u^h(x+t\vec n) - x~\mbox{d}t
\mbox{d}x \\
& = \frac{1}{e^h} \int_S f^h(x) \cdot Q^h\frac{\sqrt{e^h}}{h} V^h[y^h] ~\mbox{d}x
= \int_S f\cdot Q^h V^h[y^h], 
\end{split}
\end{equation*}
and the result follows.
\end{proof}

\medskip

\noindent {\bf Proof of Theorem \ref{thmaincinque}}.

\noindent {\bf 1.} We first show that given any $u^h\in W^{1,2}(S^h,\mathbb{R}^3)$
there exists $Q^h\in SO(3)$ and $c^h\in \mathbb{R}^3$ such that
$w^h = (Q^h)^T u^h - c^h - \mbox{id}$ satisfies:
\begin{equation}\label{a0}
\|w^h\|^2_{W^{1,2}(S^h)} \leq Ch^{-1} I^h(y^h).
\end{equation}
Indeed, by Lemma \ref{approx} and properties of the energy density $W$,
it follows that:
\begin{equation}\label{a1}
\begin{split}
I^h(y^h) &\geq Ch^{-1}\int_{S^h}\mbox{dist}^2(\nabla u^h, SO(3))
\geq Ch^{-1}\int_{S^h}|\nabla u^h - R^h\pi|^2\\
&\geq Ch^{-1}\int_{S^h} |(Q^h)^T\nabla u^h -\mbox{Id}|^2 - 
Ch^{-1}\int_{S^h}|(Q^h)^T R^h\pi -\mbox{Id}|^2\\
&\geq  Ch^{-1}\int_{S^h} |\nabla w^h|^2 - 
Ch^{-2}I^h(y^h).
\end{split}
\end{equation}
Actually, the assumption of smallness of $h^{-3} E(u^h, S^h)$ cannot be expected
to hold here.  In this general case one exchanges the $SO(3)$-valued matrix
field $R^h$ with $\tilde R^h\in W^{1,2}(S,\mathbb{R}^{3\times 3})$ given 
in the proof of Lemma \ref{approx}.  Then $Q^h$ for which the above estimates
are true may be taken as a rotation in $SO(3)$ with minimal distance from
$\fint_S \tilde R^h$.

By (\ref{a1}) it follows that:
$$\|\nabla w^h\|^2_{L^2(S^h)} \leq Ch I^h(y^h) + Ch^{-1} I^h(y^h),$$
which implies (\ref{a0}) in view of the Poincar\'e inequality,
for an appropriately chosen constant $c^h$.  A proof of the uniform Poincar\'e
inequality on $S^h$ can be found, for example, in \cite{LewMul}.

\medskip

{\bf 2.} Notice that by the definition of $m^h$ we have:
$$ \frac 1h \int_{S^h} f^h(z)\cdot Q^h z~\mbox{d}z 
= h\sqrt{e^h}\int_S\fint_{-h/2}^{h/2} f(x)\cdot Q^h (x+t\vec n) \leq m^h.$$
Therefore, in view of (\ref{fhass}) and (\ref{a0}) we obtain:
\begin{equation}\label{a2}
\begin{split}
J^h(y^h) - I^h&(y^h) = m^h - \frac{1}{h}  \int_{S^h}f^h u^h \\ 
& = - \frac{1}{h}  \int_{S^h} (Q^h)^T f^h \cdot w^h + m^h - \frac{1}{h}  
\int_{S^h} f^h\cdot Q^h z~\mbox{d}z   \\ 
&\geq - \frac{1}{h}\int_{S^h} (Q^h)^T f^h\cdot w^h 
\geq -C h^{1/2}\sqrt{e^h} \|f\|_{L^2(S)} \|w^h\|_{L^2(S^h)}\\ 
& \geq - C \sqrt{e^h} I^h(y^h)^{1/2}.
\end{split}
\end{equation}

We now prove the first claim of the theorem. 
Taking $u^h(z)=\bar{Q} z$ for any $\bar{Q}\in \mathcal{M}$, we notice that 
$J^h(y^h) = 0$. Hence $\inf J^h \leq 0$.  
The lower bound of $\frac{1}{e^h} J^h$ follows
from (\ref{a2}):
\begin{equation}\label{a3}
\frac{1}{e^h} J^h(y^h) \geq \frac{1}{e^h} I^h(y^h) - 
C \left(\frac{1}{e^h} I^h(y^h) \right)^{1/2}, 
\end{equation}
which proves (i).

\medskip

{\bf 3.} To prove (ii), let $u^h$ be a minimizing sequence of $\frac{1}{e^h} J^h$, 
as defined in (\ref{min_seq}).
Then $\{\frac{1}{e^h} J^h(y^h)\}$ is bounded, and therefore, by (\ref{a3})
$\{\frac{1}{e^h} I^h(y^h)\}$ is also bounded. The convergences of
$\tilde y^h$, $V^h[\tilde y^h]$ and $\frac{1}{h}\mbox{sym }\nabla V^h[\tilde y^h]$
follow from Theorem \ref{thmainuno}. 
In particular:
\begin{equation}\label{reza2}
\liminf_{h\to 0} \frac{1}{e^h} I^h(y^h) \geq I(V, B_{tan}).
\end{equation}
The strong convergence of $\frac{1}{h}\mbox{sym }\nabla V^h[\tilde y^h]$ 
is deduced from the strong convergence of the sequence $\mbox{sym}G^h_{tan}$ 
in Lemma \ref{lem3.6}. 
This last result is in turn implied by the convergence of $\int_S\mathcal{Q}_3(G^h)$
(valid because the sequence is minimizing), positive definiteness
of $\mathcal{Q}_3$ on symmetric matrices, and the weak convergence of $G^h$.
Since the details are exactly the same as in \cite{FJMhier} section 7.2, we omit them.

We now prove that the limit $\bar{Q}$ of any converging subsequence of $Q^h$ 
belongs to $\mathcal{M}$. By (\ref{fhass}) we have: 
\begin{equation}\label{reza3}
\begin{split}
\frac{1}{e^h}J^h(y^h) - &\frac{1}{e^h} I^h(y^h)
= \frac{1}{e^h}  \left(m^h - \frac{1}{h}\int_{S^h} f^h u^h\right)\\
& = \frac{h}{\sqrt{e^h}} \left(m - \int_{S}\fint_{-h/2}^{h/2} f(x)\cdot Q^h 
\tilde u^h ~\mbox{d}t\mbox{d}x\right)\\
& = - \frac{1}{h e^h} \int_{S^h} f^h\cdot Q^h (\tilde u^h - id)
+ \frac{h}{\sqrt{e^h}}\left(m - \int_S f\cdot Q^hx~\mbox{d}x\right).
\end{split}
\end{equation} 
The first term above is bounded, as it in fact converges to $-\int_S f\cdot \bar{Q}V$,
by Lemma \ref{Gammafh}. The quantity is brackets in the second term converges to
$m - \int_S f\cdot \bar{Q}x$.
Therefore, if $\bar{Q}\not\in\mathcal{M}$, this last quantity is uniformly positive, 
and the second term above converges to $+\infty$ (as $h/\sqrt{e^h}\to\infty$).
We observe that, in this situation, $\frac{1}{e^h} J^h(y^h)$ must converge 
to $+\infty$, contradicting (i) and thus proving that $\bar{Q}\in\mathcal{M}$.

In view of (\ref{reza2}), (\ref{reza3}) also implies:
$$ \liminf_{h\to 0} \frac1 {e^h} J^h(y^h) 
\geq \liminf_{h\to 0} \frac1{e^h} I^h(y^h) - \int_S f\cdot \bar{Q}V 
\geq J(V, B_{tan},\bar{Q}).$$ 
The fact that the limit $(V, B_{tan},\bar{Q})$ minimizes the functional $J$ is 
now a standard consequence of the above inequality. 
Indeed, if:
$$J(\hat V, \hat B_{tan},\hat{Q}) \leq J(V, B_{tan},\bar{Q}) - \epsilon$$
for some $\hat V\in W^{2,2}(S,\mathbb{R}^3)$ satisfying (\ref{Vass1})
(\ref{Vass2}), some $\hat B_{tan}\in\mathcal{B}$, 
$\hat{Q}\in\mathcal{M}$ and $\epsilon>0$, then
for a related recovery sequence $\hat y^h$ there would be:
$$\lim_{h\to 0} \frac{1}{e^h} J^h(\hat{Q}\hat y^h) 
= J(\hat V, \hat B_{tan},\hat{Q}) 
\leq J(V, B_{tan},\bar{Q}) - \epsilon \leq \liminf_{h\to 0}
\frac{1}{e^h} J^h(y^h) - \epsilon,$$
which contradicts (\ref{min_seq}). 

Finally, (iii) follows exactly as (i) and (ii). 
\endproof

\begin{remark}\label{remiforces}
{\bf 1.}  A dead load (versus a ``live load'') is any external force which only depends 
on the reference configuration point, and not on the deformation itself. 
An important feature of dead loads, discussed first in \cite{Mulunpub}, is the following. 
If the load is in a certain average sense compressive, it is advantageous 
for the body to perform a large rotation rather than undergo a compression. 
Our analysis identifies ${\mathcal M}$ as the set of candidates for such rotations,
which are expected to minimize the total energy 
$J^h$ among all rigid motions of the body. 

This phenomenon may happen even if the average torque of the force is zero:
\begin{equation}\label{torque}
\int_S f(x)\times x~\mbox{d}x = 0.
\end{equation}
Note that vanishing of the average torque is necessary  for 
$\mbox{Id} \in {\mathcal M}$, since (\ref{torque}) can be written as:
$\int_S f\cdot Fx =0$ for all $F\in so(3)$.
However, it is not sufficient, and if $\mbox{Id}\not\in \mathcal{M}$ then 
we observe that the minimizers of $J^h$ will not be close to $\mbox{Id}$.
In general, the body chooses an infinitesimal isometric displacement $V$ and 
a rotation $\bar{Q}\in\mathcal{M}$ which is energetically
advantageous in response to the force $f$. That is, those 
rotations which allow for a better alignment of infinitesimal isometries 
with the direction of the dead load, are preferred.

\medskip

{\bf 2.} The assumption on the sequence of forces $f^h$ can of course be weakened.
For example, consider $f^h(x+t\vec n)= \mbox{det}(\mbox{Id}+t\Pi(x))^{-1}
f^h(x)$ and let $\frac{1}{h\sqrt{e^h}}f^h$ converge weakly in $L^2(S)$ to some
$f\in L^2(S,\mathbb{R}^3)$. In this situation, one needs to enforce extra assumptions 
on the asymptotic behavior of the maximizers 
of the linear functions  $SO(3)\ni Q\mapsto \int_S f^h\cdot Qx~\mbox{d}x$ 
with respect to ${\mathcal M}$, to exclude certain degenerate 
cases. The analysis is as in the proof of Theorem \ref{thmaincinque} and
we leave the details to a courageous reader.

\medskip

{\bf 3.} The lower bound on $J$ and existence of its minimizers 
can be proved independently, and under the following 
weaker assumptions:
\begin{equation*}
\int_S f= 0 \quad \mbox{ and } \quad \int_S f(x)\cdot \bar{Q}Fx~\mbox{d}x=0
\quad \forall\bar{Q}\in\mathcal{M} \quad \forall F\in so(3),
\end{equation*}
which can be seen as the linearization of (\ref{setM}), although
it makes perfect sense for any closed nonempty subset $\mathcal{M}\subset SO(3)$.
Indeed, the second equality above follows by differentiating 
the expression $\int_S f(x) \cdot Qx ~\mbox{d} x$  at $\bar{Q}\in\mathcal{M}$ 
and using that $so(3)$ is the tangent space to $SO(3)$ at $\mbox{Id}$.
We present the proof of coercivity and the attainment of the minimum
by $J$ and $\tilde J$ under this condition, for arbitrary $\mathcal{M}$,
in Appendix C.
\end{remark}

\section{Appendix A - an approximation theorem on surfaces}

For a given vector mapping $u\in W^{1,2}(\Omega,\mathbb{R}^n)$ defined on 
an open subset $\Omega\subset \mathbb{R}^n$, denote:
$$E(u,\Omega) = \int_\Omega \mbox{dist}^2(\nabla u(x), SO(3))~\mbox{d}x.$$ 

\begin{lemma}\label{approx}
Let $u\in W^{1,2}(S^h,\mathbb{R}^n)$
and assume that $h^{-3} E(u, S^h)$ is sufficiently small. 
Then there exists a matrix field $R\in W^{1,2}(S,\mathbb{R}^{3\times 3})$, such that:
$$R(x) \in SO(3)\qquad \forall x\in S,$$
and a matrix $Q\in SO(3)$ with the following properties:
\begin{itemize}
\item[(i)] $\|\nabla u - R\pi\|_{L^2(S^h)} \leq CE(u, S^h)^{1/2},$
\item[(ii)] $\|\nabla R\|_{L^2(S)} \leq Ch^{-3/2} E(u, S^h)^{1/2},$
\item[(iii)] $\|Q^TR - \mathrm{Id}\|_{L^p(S)} \leq Ch^{-3/2} E(u, S^h)^{1/2},$ 
for all $p\in [1,\infty)$,
\end{itemize}
where $C$ is independent of $u$ and $h$ (but may depend on $p$).
\end{lemma}

The proof of Lemma \ref{approx} uses the following nonlinear quantitative rigidity estimate
by Friesecke, James and M\"{u}ller:
\begin{theorem}\label{approx_basic}\cite{FJMgeo}
Let $\Omega\subset \mathbb{R}^n$ be an open, bounded domain
with Lipschitz boundary.
Then, for every $u\in W^{1,2}(\Omega,\mathbb{R}^n)$ one has:
$$\min_{R\in SO(n)} \int_\Omega |\nabla u(x) - R|^2 ~\mathrm{d}x \leq
C E(u,\Omega),$$
where the constant $C$ depends only on $\Omega$.  In particular, $C$ is
invariant under dilations of $\Omega$, and it is also 
uniform for the uniform bilipschitz images of a unit ball in $\mathbb{R}^n$.
\end{theorem}

\noindent {\bf Proof of Lemma \ref{approx}.}

\noindent {\bf 1.} 
For $x\in S$ define 'balls' in $S$ and the corresponding 'cylinders' in $S^h$:
$$D_{x,h} = B(x,h)\cap S, \qquad B_{x,h} = \pi^{-1}(D_{x,h})\cap S^h.$$
The main observation is that Theorem \ref{approx_basic} may be applied on each set $B_{x,h}$,
yielding matrices $R_{x,h}\in SO(3)$ such that:
\begin{equation}\label{A:zero}
\int_{B_{x,h}} |\nabla u(z) - R_{x,h}|^2~\mbox{d}z \leq C E(u, B_{x,h}),
\end{equation}
with uniform constant $C$ (independent of $x$ or $h$).

\medskip

{\bf 2.} 
Let $\vartheta\in \mathcal{C}_c^\infty ([0,1))$ be a nonnegative cut-off function,
equal to a nonzero constant in a neighborhood of $0$.
For each $x\in S$ define the function $ \eta_x:S^h\longrightarrow \mathbb{R}$:
$$\displaystyle\eta_x(z) = \frac{\vartheta(|\pi z - x|/h)}
{\int_{S^h}\vartheta(|\pi y - x|/h) ~\mbox{d}y}.$$
Then $\eta_x(z) = 0$ for $z\not\in B_{x,h}$ and:
\begin{equation}\label{A:eta_prop}
\int_{S^h} \eta_x(z)~\mbox{d}z = 1,\qquad \|\eta_x\|_{L^\infty}\leq {C}h^{-3},
\qquad \|\nabla_{x}\eta_x\|_{L^\infty}\leq {C}h^{-4}.
\end{equation}
The last inequalities follow from the lipschitzianity of $\partial S$.
In particular, the denominator function in the definition of $\eta_x$
has Lipschitz constant $Ch^2$, and hence:
$$\left\| \nabla_x \left(\int_{S^h}\vartheta(|\pi y - x|/h) ~\mbox{d}y\right)^{-1}\right\|_{L^\infty}
\leq {C} h^{-4}.$$
Consider the matrix field $\tilde R\in W^{1,2}(S,\mathbb{R}^{3\times 3})$:
$$\tilde R(x) = \int_{S^h} \eta_x(z) \nabla u(z)~\mbox{d}z.$$
By the first two statements in (\ref{A:eta_prop}) we obtain:
\begin{equation}\label{A:uno}
|\tilde R(x) - R_{x,h}|^2 = 
\left|\int_{S^h} \eta_x(z) (\nabla u(z)- R_{x,h})~\mbox{d}z\right|^2
\leq {C} h^{-3}E(u, B_{x,h}).
\end{equation}
Similarly:
\begin{equation}\label{A:due}
\begin{split}
|\nabla \tilde R(x)|^2 &= \left|\int_{S^h}(\nabla_x\eta_x)\nabla u\right|^2
= \left|\int_{S^h}(\nabla_x\eta_x)(\nabla u- R_{x,h})\right|^2\\
& \leq \int_{B_{x,h}}\left|\nabla_x\eta_x\right|^2\cdot
\int_{B_{x,h}}\left|\nabla u - R_{x,h}\right|^2 \leq {C}h^{-5}
E(u, B_{x,h}),
\end{split}
\end{equation}
and likewise, for any $x'\in D_{x,h}$:
\begin{equation}\label{A:due.5}
|\nabla \tilde R(x')|^2 \leq {C}h^{-5} E(u, 2B_{x,h})
\end{equation}
with $2B_{x,h} = \pi^{-1}(D_{x,2h})\cap S^h$. 
Therefore, in view of the lipschitzianity of $\partial S$ and by the fundamental theorem
of calculus:
\begin{equation}\label{A:tre}
|\tilde R(x'') - \tilde R(x)|^2 \leq {C}h^{-3}E(u, 2B_{x,h})\qquad \forall x''\in D_{x,h}.
\end{equation}
Combining (\ref{A:zero}) with (\ref{A:uno}) and (\ref{A:tre}) yields:
\begin{equation}\label{A:quattro}
\begin{split}
&\int_{B_{x,h}} |\nabla u(z)  - \tilde R \pi (z)|^2~\mbox{d}z \\
&\quad \leq  2\left( \int_{B_{x,h}} |\nabla u - R_{x,h}|^2 + \int_{B_{x,h}} |\tilde R(x)- R_{x,h}|^2
+ \int_{B_{x,h}} |\tilde R\pi - \tilde R(x)|^2 \right)\\
&\quad \leq C E(u, 2B_{x,h}).
\end{split}
\end{equation}
Now cover $S$ by $\{D_{x_i,h}\}_{i=1}^{N_h}$ so that the covering number of 
the family $\{2B_{x_i,h}\}_{i=1}^{N_h}$ is independent of $h$. Summing the inequalities in 
(\ref{A:quattro}) over $i=1\ldots N$ proves:
\begin{equation}\label{A:cinque}
\int_{S^h} |\nabla u - \tilde R\pi|^2 \leq C E(u, S^h).
\end{equation}
Also, integrating (\ref{A:due.5}) over $D_{x_i, h}$ and summing over $i=1\ldots N$ gives:
\begin{equation}\label{A:sei}
\int_{S} |\nabla \tilde R|^2 \leq {C}h^{-3} E(u, S^h).
\end{equation}

\medskip

{\bf 3.} Notice that by (\ref{A:uno}), for every $x\in S$:
$$\mbox{dist}^2(\tilde R(x), SO(3))\leq {C}h^{-3} E(u,S^h).$$
Hence, if $E(u, S^h)/h^3$ is sufficiently small, we may define:
$$R(x) = \mathcal{P}_{SO(3)} (\tilde R(x))$$
where $\mathcal{P}_{SO(3)}$ is the orthogonal projection onto the compact manifold $SO(3)$. 
Clearly $R:S\longrightarrow SO(3)$ is a $W^{1,2}$ matrix field and since:
$$|R(x) - \tilde R(x)| = \mbox{dist }(\tilde R(x), SO(3)),$$
by (\ref{A:cinque}) we conclude that:
\begin{equation*}
\int_{S^h}|\nabla u - R\pi|^2 \leq
C\left( \int_{S^h}|\nabla u - \tilde R\pi|^2 
+ \int_{S^h}\mbox{dist}^2(\nabla u, SO(3))\right)
\leq C E(u, S^h),
\end{equation*}
which proves (i) in Lemma \ref{approx}. The bound (ii) is deduced directly from 
(\ref{A:sei}).

\medskip

{\bf 4.} To deduce (iii), define first the intermediate matrix $\tilde Q$
as the average of $R$ on $S$. Using the Sobolev and Poincar\'e inequalities,
together with (ii) we obtain, for every $p\geq 2$:
\begin{equation}\label{A:sette}
\left(\int_S |R-\tilde Q|^p\right)^{2/p} \leq 
C \|R-\tilde Q\|^2_{W^{1,2}(S)}\leq C\int_S |\nabla R|^2 \leq {C}h^{-3} E(u, S^h).
\end{equation} 
Now, take $Q\in SO(3)$ such that $|Q - \tilde Q| = \mbox{dist }(\tilde Q, SO(3))$.
As before, (\ref{A:sette}) remains true if $\tilde Q$ is replaced with $Q$.
Clearly, the same bound must also hold for $p\in [1,2)$, and so we conclude that:
$$\forall p\in [1,\infty) \qquad \|R - Q\|^2_{L^p(S)} \leq {C}h^{-3} E(u, S^h).$$
The above easily implies (iii).
\endproof

\section{Appendix B - the $\Gamma$-convergence setting}

We first recall the notion of $\Gamma$-convergence of a sequence of functionals 
$\mathcal{F}^h:X\longrightarrow \overline{\mathbb{R}}$, defined on a metric
space $X$. Namely, $\mathcal{F}^h$ $\Gamma$-converge, as $h\to 0$, to some 
$\mathcal{F}:X\longrightarrow \overline{\mathbb{R}}$ provided that 
the following two conditions hold:
\begin{itemize}
\item[(i)] For any converging sequence $\{x^h\}$ in $X$ one has:
\begin{equation}\label{Gamma1}
\mathcal{F}\left(\lim_{h\to 0} x^h\right) \leq \liminf_{h\to 0} \mathcal{F}^h(x^h).
\end{equation}
\item[(ii)] For every $x\in X$, there exists a sequence $\{x^h\}$ converging
to $x$, such that:
\begin{equation}\label{Gamma2}
\mathcal{F}(x) = \lim_{h\to 0} \mathcal{F}^h(x^h).
\end{equation}
\end{itemize}
When $X$ is only a topological space, the definition of $\Gamma$-convergence 
involves, naturally, systems of
neighborhoods rather than sequences.  However, when the functionals $\mathcal{F}^h$
are equi-coercive and $X$ is a reflexive Banach space equipped with 
weak topology, one can still use (i) and (ii) above (for weakly converging
sequences), as an equivalent version of this definition.
For details, we refer the reader to \cite{dalmaso}.

\medskip

\noindent {\bf Proof of Corollary \ref{thmainquattro}}.

\noindent We only prove (i), in the case when the product space in the domain of 
$\mathcal{F}$ is equipped with the strong topology.
The other statements follow the same.

To obtain (\ref{Gamma1}), we take a sequence of $W^{1,2}(S^{h_0})$ vector mappings
$\{y^h\}$ such that, writing $B^h_{tan} = \frac{1}{h} \mbox{sym }\nabla V^h[y^h]$,
the sequence $\{\mathcal{F}^h(y^h, V^h[y^h], B^h_{tan})\}$
is bounded, and such that $y^h$, $V^h[y^h]$ and $B_{tan}^h$ 
converge to some $y$, $V$ and $B_{tan}$ (in $W^{1,2}(S^{h_0})$, 
$W^{1,2}(S)$ and $L^2(S)$ respectively).
By Theorem \ref{thmainuno} we obtain a sequence of normalized
deformations $\tilde y^h = (Q^h)^T y^h - c^h$, converging to $\pi$.
Moreover, $V^h[\tilde y^h]$ and $\frac{1}{h}\mbox{sym }\nabla V^h[\tilde y^h]$
converge to $\tilde V$ and weakly to $\tilde B_{tan}$, respectively.
Notice now that:
$$|Q^h - \mbox{Id}|= C h^{-1}\sqrt{e^h} \big\|\nabla V^h[\tilde y^h]
- (Q^h)^T \nabla V^h[y^h]\big\|_{L^2(S)} \leq Ch^{-1}\sqrt{e^h}.$$ 
In particular, Lemma \ref{appr} remains true if we put $Q^h = \mbox{Id}$,
for all $h$.  Consequently, all the assertions of Theorem \ref{thmainuno}
still hold for $\tilde y^h = y^h - c^h$ (possibly after 
modifying the constants $c^h$).

Now, $V^h[y^h] - V^h[\tilde y^h] = h/\sqrt{e^h} c^h$ is bounded, so
$c^h$ converge to $0$, as $h\to 0$. On the other hand
$c^h = y^h - \tilde y^h$ converge to $y-\pi$.  Hence $y=\pi$.
Moreover $\nabla V^h[\tilde y^h] = \nabla V^h[y^h]$, so $\nabla V=\nabla \tilde V$
and $B_{tan} = \tilde B_{tan}$.  
By Theorem \ref{thmainuno} (iv) we conclude that:
$$\mathcal{F}(y,V,B_{tan}) \leq \liminf_{h\to 0} \mathcal{F}^h(y^h, V^h, B^h_{tan})$$
which proves (\ref{Gamma1}).

The second requirement for $\Gamma$-convergence (\ref{Gamma2}) follows 
directly from Theorem \ref{thmaindue}, in view of (\ref{Gamma1}).
\endproof

We remark that in presence of external forces, the results of Theorem \ref{thmaincinque}
can also be formulated in the language of $\Gamma$-convergence, similarly as above.

\section{Appendix C - on coercivity of the generalized von K\'arm\'an functionals 
$J$ and $\tilde J$}

In this section, we consider the functionals:
\begin{equation*}
\begin{split}
J(V, B_{tan},\bar{Q}) =& \frac{1}{2} 
\int_S \mathcal{Q}_2\left(x,B_{tan} - \frac{\kappa}{2} (A^2)_{tan}\right)\\
& + \frac{1}{24} \int_S \mathcal{Q}_2\left(x,(\nabla(A\vec n) - A\Pi)_{tan}\right)
- \int_S f\cdot \bar{Q}V,\\
\tilde J(V,\bar{Q}) =& \frac{1}{24} \int_S \mathcal{Q}_2\left(x,(\nabla(A\vec n) - A\Pi)_{tan}\right)
- \int_S f\cdot\bar{Q} V,
\end{split}
\end{equation*}
defined for infinitesimal isometries $V$, matrix fields $B_{tan}\in\mathcal{B}$
and rotations $\bar{Q}\in\mathcal{M}$, where $\mathcal{M}$ is an arbitrary
closed and nonempty subset of $SO(3)$.
We prove that $J$ and $\tilde J$ attain their finite lower bounds under the 
following assumptions on $f\in L^2(S,\mathbb{R}^3)$:
\begin{equation}\label{sette.cinque}
\int_S f = 0 \quad \mbox{ and } \quad 
\int_S f(x)\cdot \bar{Q}Fx~\mbox{d}x = 0 \quad\forall \bar{Q}\in\mathcal{M}
\quad\forall F\in so(3).
\end{equation}
As mentioned in Remark \ref{remiforces}, the second condition above is a 
consequence and linearization of (\ref{setM}), with $\mathcal{M}$ defined by
that formula.

\begin{lemma}\label{Jcoercivity}
Assume that $S$ is of class $\mathcal{C}^{2,1}$.  Then for every 
$V\in W^{2,2} (S,\mathbb{R}^3)$ satisfying (\ref{Vass1}) and (\ref{Vass2})
there exist $D\in so(3)$ and $d\in\mathbb{R}^3$, so that:
$$\|V - (Dx+d)\|^2_{W^{2,2}(S)} \leq C 
\int_S \left|(\nabla(A\vec n) - A\Pi)_{tan}\right|^2.$$
\end{lemma}
\begin{proof}
{\bf 1.} We first prove that $\int_S \left|(\nabla(A\vec n) - A\Pi)_{tan}\right|^2 = 0$
implies for a $W^{2,2}$ infinitesimal isometry $V$ to have the form $V(x)=Dx+d$,
with $D\in so(3)$ and $d\in\mathbb{R}^3$.

To see this, let $c\in W^{1,2}(S,\mathbb{R}^3)$ be such that:
$$A(x)\tau = c(x)\times \tau \qquad \forall x\in S \quad \forall \tau\in T_x S.$$
Since $A$ represents a gradient, it follows
that $\partial_\tau c\times \eta =\partial_\eta c\times \tau$ 
for all $\tau,\eta\in T_x S$.
In particular, for any $\tau$ and $\eta$ such that $\tau\times\eta = \vec n$,
one has:
$$(\partial_\tau c)\cdot \vec n 
= - (\partial_\tau c\times \eta)\cdot \tau 
= - (\partial_\eta c\times \tau)\cdot \tau = 0.$$
On the other hand:
$$0 = \big(\partial_\tau (A\vec n) - A\Pi\tau\big)_{tan} 
= \big(\partial_\tau (c\times \vec n) - A\Pi\tau\big)_{tan} 
= (\partial_\tau c)\times\vec n.$$ 
Hence $\partial_\tau c = 0$ on $S$, which yields the claim.

\medskip

{\bf 2.} We prove the result. Arguing by contradiction, we assume that for a sequence of
infinitesimal isometries $V^h\in W^{2,2}(S,\mathbb{R}^3)$ there holds:
\begin{equation}\label{C.due}
\begin{split}
\mbox{dist}_{W^{2,2}(S)}&\Big(V^h, \{Dx+d; ~ D\in so(3),~ d\in\mathbb{R}^3\}\Big)=1\\
&\qquad\qquad \qquad\mbox{ and } \quad
\lim_{h\to 0} \int_S \left|(\nabla A^h)\vec n\right|^2 = 0.
\end{split}
\end{equation}
Since the second condition above involves only higher derivatives of $V^h$, we
may without loss of generality replace the first condition by:
\begin{equation}\label{C.tre}
\|V^h\|_{W^{2,2}(S)}=1 \quad \mbox{ and } \quad
\big\langle V^h, Dx+d\big\rangle_{W^{2,2}(S)} = 0 \quad 
\forall D\in so(3), ~ d\in \mathbb{R}^3.
\end{equation}
In particular, $V^h$ converges weakly in $W^{2,2}(S)$ (up to a subsequence, which we 
do not relabel) to some 
vector field $V$, still satisfying (\ref{Vass1}) and (\ref{Vass2}).
By (\ref{C.due}) and the weak lower semicontinuity of the $L^2$ norm,
we deduce that $\int_S\left|(\nabla A)\vec n\right|^2 = 0$.  Hence, in view of
the first part of the proof and the second condition in (\ref{C.tre}), it follows
that $V=0$ and so:
\begin{equation}\label{C.quattro}
\lim_{h\to 0} \|V^h\|_{W^{1,2}(S)}=0.
\end{equation}

By the estimate (\ref{splitting}) and (\ref{C.quattro}) we may deduce:
\begin{equation}\label{C.cinque}
\lim_{h\to 0}\|V^h_{tan}\|_{W^{2,2}(S)}=0,
\end{equation}
where $V^h_{tan} = V^h - (V^h\vec n)\vec n$.
Observe that:
\begin{equation*}
\begin{split}
\int_S\left|(\nabla A^h)\vec n\right|^2 
&= \int_S\Big|\nabla \left(\Pi V^h_{tan}-\nabla (V^h\vec n)\right) - A^h\Pi\Big|^2\\
&= \int_S\Big|\nabla^2 (V^h\vec n) + (A^h\Pi - \Pi A^h)_{tan} 
- (\nabla \Pi) V^h_{tan} + (V^h\vec n) \Pi\Big|^2.
\end{split}
\end{equation*}
Therefore:
\begin{equation*}
\|\nabla^2(V^h\vec n)\|_{L^2(S)}\leq C\Big(\|(\nabla A^h)\vec n\|_{L^2(S)}
+ \|V^h\|_{W^{1,2}(S)}\Big),
\end{equation*}
and in view of (\ref{C.quattro}) and the assumption (\ref{C.due}) we also get:
$$\lim_{h\to 0} \|V^h \vec n\|_{W^{2,2}(S)}=0.$$
The above together with (\ref{C.cinque}) contradicts (\ref{C.tre}) and proves the lemma.
\end{proof}

\begin{lemma}
Assume (\ref{sette.cinque}) and let $S$ be of class $\mathcal{C}^{2,1}$. 
Then the functionals $J$ and $\tilde J$, defined for $V\in W^{2,2}(S,\mathbb{R}^3)$
satisfying (\ref{Vass1}), (\ref{Vass2}), and $B_{tan}\in\mathcal{B}$,
$\bar{Q}\in\mathcal{M}$, are bounded from below and attain their infima.
\end{lemma}
\begin{proof}
{\bf 1.}
Let $V\in W^{2,2}(S,\mathbb{R}^3)$ be an infinitesimal isometry.
By Lemma \ref{Jcoercivity}, positive definiteness of $\mathcal{Q}_2$ (on symmetric matrices)
and (\ref{sette.cinque}), we obtain:
\begin{equation}\label{C.sei}
\begin{split}
\tilde J(V) &\geq C \|\tilde V\|^2_{W^{2,2}(S)} - \int_S f\cdot\bar{Q}V
= C \|\tilde V\|^2_{W^{2,2}(S)} - \int_S f\cdot\bar{Q}\tilde V\\
&\geq   C \|\tilde V\|^2_{W^{2,2}(S)} - \|f\|_{L^2(S)} \cdot \|\tilde V\|_{L^2(S)},
\end{split}
\end{equation}
for an appropriate modification $\tilde V = V - (Dx+d)$.
Hence the lower bound on $J$ (and $\tilde J$) follows.

\medskip

{\bf 2.}
Let now $(V^h, B^h_{tan}, \bar{Q}^h)$ be a minimizing sequence of $J$.  
Clearly, a subsequence of $\bar{Q}^h$ converges to some $\bar{Q}\in\mathcal{M}$.

Using (\ref{C.sei}) and applying the positive definiteness of $\mathcal{Q}_2$ to the first term in $J$, 
there follows the (uniform in $h$) boundedness of the following expressions:
\begin{equation}\label{C.sette}
\Big(C\|\tilde V^h\|^2_{W^{2,2}(S)} - \|f\|_{L^2(S)}\cdot \|\tilde V^h\|_{L^2(S)}\Big)
+ C \left\|B^h_{tan} - \frac{\kappa}{2} ((A^h)^2)_{tan}\right\|^2_{L^2(S)}.
\end{equation}
Again, we put $\tilde V^h = V^h - (D^h x + d^h)$ and apply Lemma \ref{Jcoercivity}.
In particular, the sequence $\tilde V^h$ is bounded in $W^{2,2}(S)$ and so it converges
(up to a subsequence), weakly in $W^{2,2}(S)$, to an infinitesimal isometry $V$.
Further, the matrix fields $\tilde A^h = A^h - D^h$ converge weakly in $W^{1,2}(S)$
to the field $A$ satisfying (\ref{Vass1}) and (\ref{Vass2}).

Notice that:
$$(A^h)^2 = (\tilde A^h)^2 + (D^h)^2 + (D^hA^h + A^hD^h).$$
Hence the boundedness of the second term in (\ref{C.sette}) results in the
$L^2(S)$ boundedness of:
$$B^h_{tan} - \frac{\kappa}{2}\left((D^h)^2 + (D^hA^h + A^hD^h)\right)_{tan}
= B^h_{tan} - \frac{\kappa}{2} \mbox{sym }\nabla \left( (D^h)^2 x + 2D^h V^h(x)\right).$$
We may now conclude that a subsequence of the above sequence of symmetric 
matrix fields converges, weakly in $L^2(S)$, to some $B_{tan}\in\mathcal{B}$.  
Thus, $B^h_{tan} - \frac{\kappa}{2} ((A^h)^2)_{tan}$ converges to 
$B_{tan} - \frac{\kappa}{2}(A^2)_{tan}$.

By the weak lower semicontinuity of both quadratic terms in $J$ we conclude that 
$J(V, B_{tan}\bar{Q})$ realizes the infimum of $J$.
Likewise, $\tilde J(V,\bar{Q})$ realizes the infimum of $\tilde J$, had $V^h$ been 
a minimizing sequence of $\tilde J$.
\end{proof}


\begin{thebibliography}{9999}

\bibitem{ciarbookvol3} P.G. Ciarlet, \textit{Mathematical Elasticity, Vol 3:
Theory of Shells}, North-Holland, Amsterdam (2000).

\bibitem{Co03} S. Conti, \textit{Habilitation thesis}, University of Leipzig, 2003.

\bibitem{CD1} S. Conti and G. Dolzmann, \textit{Derivation of elastic theories for thin sheets 
and the constraint of incompressibility}, Analysis, modeling and simulation of multiscale problems,
225--247, Springer, Berlin, 2006.

\bibitem{CD2} S. Conti and G. Dolzmann, \textit{Derivation of a plate theory 
for incompressible materials}, C. R. Math. Acad. Sci. Paris 344 (2007), no. 8, 541--544.

\bibitem{CM05} S. Conti and F. Maggi, \textit{Confining thin sheets and folding paper},
Arch. Ration. Mech. Anal.  187  (2008),  no. 1, 1--48. 

\bibitem{dalmaso} G. Dal Maso, \textit{An introduction to $\Gamma$-convergence.}, 
Progress in Nonlinear Differential Equations and their Applications, {\bf 8}, 
Birkh\"auser, MA, (1993).

\bibitem{Ev} L.C. Evans,  \textit{Partial differential equations}, 
Graduate Studies in Mathematics, {\bf 19}, American Mathematical Society, 
Providence, RI, (1998).

\bibitem{FJMM_cr} G. Friesecke, R. James, M.G. Mora and S. M\"uller, 
\textit{Derivation of nonlinear bending theory for shells from three-dimensional 
nonlinear elasticity by Gamma-convergence},  C. R. Math. Acad. Sci. Paris,  
{\bf 336}  (2003),  no. 8, 697--702.

\bibitem{FJMgeo} G. Friesecke, R. James and S. M\"uller, 
\textit{A theorem on geometric rigidity and the derivation of nonlinear 
plate theory from three dimensional elasticity}, Comm. Pure. Appl. Math., 
{\bf 55} (2002), 1461--1506.  

\bibitem{FJMhier} G. Friesecke, R. James and S. M\"uller, \textit{A hierarchy 
of plate models derived from nonlinear elasticity by gamma-convergence}, 
Arch. Ration. Mech. Anal.,  {\bf 180}  (2006),  no. 2, 183--236.

\bibitem{GSP} G. Geymonat and E. Sanchez-Palencia, \textit{On the rigidity 
of certain surfaces with folds and applications to shell theory},  
Arch. Ration. Mech. Anal., {\bf 129}  (1995),  no. 1, 11--45. 


\bibitem{HH} Q. Han and J.-X. Hong, \textit{Isometric embedding of Riemannian 
manifolds in Euclidean spaces}, Mathematical Surveys and Monographs, {\bf 130} 
American Mathematical Society, Providence, RI (2006).

\bibitem{Ho} L. H\"ormander, \textit{Uniqueness theorems for second order elliptic differential 
equations}, Comm. Partial Differential Equations, {\bf 8} (1983), no. 1, 
21--64. 

\bibitem{karman} T. von K\'arm\'an, \textit{Festigkeitsprobleme im Maschinenbau}, 
in Encyclop\"adie der Mathematischen Wissenschaften. Vol. IV/4, pp. 311-385, Leipzig,
1910.

\bibitem{LR1} H. LeDret and A. Raoult, The nonlinear membrane model as a variational 
limit of nonlinear three-dimensional elasticity. \textit{J. Math. Pures Appl.\/} 
\textbf{73} (1995), 549--578.

\bibitem{LeD-Rao} H. Le Dret and A. Raoult, \textit{The membrane shell model in nonlinear 
elasticity: a variational asymptotic derivation}, J. Nonlinear Sci., {\bf 6} (1996), 59--84.

\bibitem{LewMul} M. Lewicka and S. M\"uller, \textit{The uniform Korn-Poincar\'e 
inequality in thin domains}, submitted.

\bibitem{Liu50} F.C. Liu, \textit{A Lusin property of Sobolev functions},
Indiana U. Math. J., {\bf 26}  (1977),  645--651.

\bibitem{Love} A.E.H. Love, \textit{A treatise on the mathematical theory 
of elasticity}, 4th ed.\ Cambridge University Press, Cambridge (1927).


\bibitem{Mo03} R. Monneau, \textit{Justification of nonlinear Kirchhoff-Love theory of plates
as the application of a new singular inverse method},
Arch. Rational Mech. Anal. \textbf{169} (2003), 1--34.

\bibitem{Mulunpub} S. M\"uller, Unpublished note.

\bibitem{MuPa} S. M\"uller and M. R. Pakzad, \textit{Convergence of equilibria 
of thin elastic plates -- the von K\'arm\'an case}, to appear in Comm. Partial 
Differential Equations. 


\bibitem{Ni} L. Nirenberg, \textit{The Weyl and Minkowski problems 
in differential geometry in the large}, Comm. Pure Appl. Math., {\bf 6} (1953), 337--394.



\bibitem{Pa03} O. Pantz, \textit{On the justification of the nonlinear inextensional plate model},
Arch. Rational Mech. Anal. \textbf{167} (2003), 179--209.

\bibitem{sanchez} {\' E}. Sanchez-Palencia,
\textit{Statique et dynamique des coques minces. II. Cas de flexion pure inhibée. 
Approximation membranaire.}  C. R. Acad. Sci. Paris S\'er. I Math.  309  (1989),  
no. 7, 531--537. 

	

 
\bibitem{Schmidt} B. Schmidt, \textit{Plate theory for stressed heterogeneous
multilayers of finite bending energy}, J. Math. Pures Appl. \textbf{88} (2007), 107--122.

\bibitem{Spivak} M. Spivak, \textit{A Comprehensive Introduction to Differential 
Geometry, Vol V}, 2nd edition, Publish or Perish Inc. (1979).

\bibitem{T} K. Trabelsi, \textit{Modeling of a membrane for nonlinearly elastic incompressible 
materials via gamma-convergence}, Anal. Appl. (Singap.) \textbf{4} (2006), no. 1, 31--60.


\end{thebibliography}
\end{document}